\documentclass{amsart}

\title[{Free transformations of $S^1 \times S^n$ of square-free odd period}]{Free transformations of $S^1 \times S^n$\\ of square-free odd period}

\author[Q.~Khan]{Qayum Khan}
\address{Department of Mathematics, Saint Louis University, St. Louis MO 63103 U.S.A.}
\email{khanq@slu.edu}

\usepackage{amsthm,amsmath,amssymb,txfonts}
\usepackage[all]{xy} % only required for one diagram
\usepackage{amscd}
\usepackage{tikz-cd}\usetikzlibrary{arrows}
\usepackage{verbatim,url}

\usepackage{epigraph}
\usepackage{etoolbox}
\makeatletter
\patchcmd{\epigraph}{\@epitext{#1}}{\itshape\@epitext{#1}}{}{}
\makeatother

\usepackage{xcolor}
\definecolor{dark-red}{rgb}{0.4,0.15,0.15}
\definecolor{dark-blue}{rgb}{0.15,0.15,0.4}
\definecolor{medium-blue}{rgb}{0,0,0.5}
\usepackage[colorlinks, linkcolor={dark-red}, citecolor={dark-blue}, urlcolor={medium-blue}]{hyperref} % must be the last package loaded

\newtheorem{thm}{Theorem}[section]
\newtheorem{cor}[thm]{Corollary}
\newtheorem{lem}[thm]{Lemma}
\newtheorem{prop}[thm]{Proposition}
\newtheorem*{unthm}{Theorem}
\theoremstyle{definition}

\newtheorem{rem}[thm]{Remark}

\newtheorem{quest}[thm]{Question}

\DeclareMathAlphabet{\matheurm}{U}{eur}{m}{n}

\newcommand{\cA}{\mathcal{A}}

\newcommand{\F}{\mathbb{F}}

\newcommand{\cH}{\mathcal{H}}
\newcommand{\bL}{\mathbb{L}}
\newcommand{\C}{\mathbb{C}}
\newcommand{\cM}{\mathcal{M}}

\newcommand{\Q}{\mathbb{Q}}
\newcommand{\R}{\mathbb{R}}
\newcommand{\RP}{\mathbb{RP}}
\newcommand{\cS}{\mathcal{S}}
\newcommand{\Z}{\mathbb{Z}}

\newcommand{\Nil}{\mathrm{Nil}}

\newcommand{\TOP}{\mathrm{TOP}}

\newcommand{\longra}{\longrightarrow}
\newcommand{\hookra}{\hookrightarrow}
\newcommand{\thra}{\twoheadrightarrow}
\newcommand{\iso}{\cong}
\newcommand{\bdry}{\partial}
\newcommand{\eps}{\varepsilon}

\newcommand{\gens}[1]{\left\langle #1 \right\rangle}

\newcommand{\ol}[1]{\overline{#1}}

\newcommand{\wh}[1]{\widehat{#1}}
\newcommand{\wt}[1]{\widetilde{#1}}
\newcommand{\xra}[1]{\xrightarrow{#1}}

\newcommand{\Torus}{\mathrm{Torus}}
\newcommand{\Cok}{\mathrm{Cok}}
\newcommand{\Cl}{\mathrm{Cl}}
\newcommand{\Wh}{\mathrm{Wh}}
\newcommand{\id}{\mathrm{id}}
\newcommand{\rel}{\mathrm{rel}~}
\newcommand{\tors}{\mathrm{tors}}
\newcommand{\Ker}{\mathrm{Ker}}
\newcommand{\Img}{\mathrm{Im}}
\newcommand{\hMod}{\mathrm{hMod}}
\newcommand{\Homeo}{\mathrm{Homeo}}
\newcommand{\Map}{\mathrm{Map}}

\newcommand{\Aut}{\mathrm{Aut}}
\newcommand{\Out}{\mathrm{Out}}
\newcommand{\Diff}{\mathrm{Diff}}
\newcommand{\Hom}{\mathrm{Hom}}
\newcommand{\Isom}{\mathrm{Isom}}

\begin{document}

\begin{abstract}
Let $n$ be a positive integer, and let $\ell>1$ be square-free odd.
We classify the set of equivariant homeomorphism classes of free $C_\ell$-actions on the product $S^1 \times S^n$ of spheres, up to indeterminacy bounded in $\ell$.
The description is expressed in terms of number theory.

The techniques are various applications of surgery theory and homotopy theory, and we perform a careful study of $h$-cobordisms.
The $\ell=2$ case was completed by B~Jahren and S~Kwasik (2011).
The new issues for the case of $\ell$ odd are the presence of nontrivial ideal class groups and a group of equivariant self-equivalences with quadratic growth in $\ell$.
The latter is handled by the composition formula for structure groups of A~Ranicki (2009).
\end{abstract}

\maketitle

%------------------------------------------------------------------------------
\section{Introduction}

Let $\ell>1$ be an integer.
Consider the $\ell$-periodic homeomorphism without fixed points:
\[
T_\ell: S^1 \times S^n \longra S^1 \times S^n ~;~ (z,x) \longmapsto (\zeta_\ell z, x) \enspace\text{ where }\enspace \zeta_\ell := e^{i\, 2\pi/\ell} \in \C.
\]
Write $\cA_\ell^n$ for the set of conjugacy classes $(C)$ in $\Homeo(S^1\times S^n)$ of those cyclic subgroups $C$ of order $\ell$ without fixed points.
B~Jahren and S~Kwasik classified the case $\ell=2$ \cite{JK}.

Recall the Euler totient function $\varphi$ is the number of units modulo a given natural number.
Let $d>1$.
A partition $Q_d^k$ of $\Z_d^\times$ is given by $[q]=[q']$ if $a^k q \equiv \pm q' \!\!\pmod{d}$ for some $a$.
The map $(g \longmapsto g^{-1})$ on the cyclic group $C_d$ induces an involution $\iota$ on the projective class group $\Wh_0(C_d) := K_0(\Z C_d)/K_0(\Z)$ with coinvariants  $H_0(C_2;\Wh_0(C_d)) := \Wh_0(C_d)/(1-\iota)$.

\begin{thm}[Classification Theorem]\label{thm:classification}
Let $\ell > 1$ be square-free odd.
Then $\cA_\ell^{2k} = \{(T_\ell)\}$ for all $k>0$ and $\cA_\ell^1 = \{(T_\ell)\}$.
Otherwise for each $k>1$, there is a finite-to-one surjection
\[\begin{tikzcd}
\displaystyle \coprod_{1 < d \,|\, \ell}\; Q_d^k ~\times~ \Z^{(d-1)/2} ~\times~ H_0(C_2;\Wh_0(C_d))
\arrow[two heads]{r} & \cA_\ell^{2k-1} - \{(T_\ell)\}.
\end{tikzcd}\]
The $d$-indexed terms in the disjoint union have disjoint images.
In the $d$-th image, each point-preimage has cardinality dividing $8 \gcd(k,\varphi(d)/2)$, which has bounded growth in $\ell$.
In particular, the set $\cA_\ell^{2k-1}$ of free $C_\ell$-actions on $S^1 \times S^{2k-1}$ is countably infinite if $k>1$.
\end{thm}

Different preimages have different cardinalities (\ref{rem:cardinality}).
For $n=3$, this theorem answers the existence part of \cite[Problem 6.14]{CONFERENCE}; indeterminacy in the uniqueness is at most 16.

\begin{cor}\label{cor:classification}
Let $p \neq 2$ be prime.
Then $\cA_p^{2k} = \{(T_p)\}$ for all $k>0$ and $\cA_p^1 = \{(T_p)\}$.
Otherwise for any given $k>1$, there is a finite-to-one surjection
\[\begin{tikzcd}
\displaystyle Q_p^k ~\times~ \Z^{(p-1)/2} ~\times~ H_0(C_2;\Cl_p)
\arrow[two heads]{r} & \cA_p^{2k-1} - \{(T_p)\}.
\end{tikzcd}\]
Each preimage has cardinality dividing $8 \gcd(k,(p-1)/2)$, which is bounded in $p$.
\qed
\end{cor}

Here, $\Cl_p$ is the ideal class group of $\Z[\zeta_p]$; the involution $\iota$ is induced by $(\zeta_p \longmapsto \zeta_p^{-1})$.
The three parts are understood using the quotient manifold $M$ of the free $C_p$-action, specifically invariants of the infinite cyclic cover $\ol{M}$, as follows.
The $Q_p^k$-part is the first Postnikov invariant of $\ol{M}$.
The $\Z^{(p-1)/2}$-part is a projective $\rho$-invariant of $\ol{M}$.
The $\Cl_p$-part is the Siebenmann end obstruction of $\ol{M}$.
The indeterminacy $8 \gcd(k,(p-1)/2)$ is due to ineffective action of the group (quadratic growth in $p$) of self-homotopy equivalences of $M$.

\begin{rem}\label{rem:class}
Consider the ideal class group $\Cl_p^+$ of the real subring $\Z[\zeta_p+\zeta_p^{-1}]$ of $\Z[\zeta_p]$.
Write $G$ for the Galois group of $\Q(\zeta_p)$ over $\Q$.
The induced $\Z[G]$-module map $\Cl_p^+ \longra \Cl_p$ is injective (\cite[Theorem~4.14]{Washington_book}).
The norm map $N := 1 + \iota : \Cl_p \longra \Cl_p^+$ is surjective (\cite[Proof~10.2]{Washington_book}).
Since the fixed field of the automorphism $\iota \in G$ is $\Q(\zeta_p+\zeta_p^{-1})$, $\iota$ induces the identity on $\Cl_p^+$.
Then $\iota$ induces negative the identity on $\Cl_p^- := \Cl_p/\Cl_p^+$, since
\[
\iota(I) ~~= N(I) - I ~\equiv~ -I \pmod{\Cl_p^+}.
\]
Therefore we obtain an exact sequence of $\Z[G/\iota]$-modules:
\[\begin{tikzcd}
{}_2(\Cl_p^-) \rar{\wt{1-\iota}} & \Cl_p^+ \rar & H_0(C_2;\Cl_p) \rar & \Cl_p^-/2 \rar & 0.
\end{tikzcd}\]
Here ${}_2 A := \{a \in A \mid 2a=0\}$ denotes the exponent-two subgroup of any abelian group $A$, and $\wt{1-\iota} := (1-\iota)\circ s$ is a well-defined homomorphism via a setwise section $s: \Cl_p^- \longra \Cl_p$.
\end{rem}

\begin{rem}\label{rem:class_comp}
The $\Cl_p^-/2$ are only known for $p<500$ \cite{Schoof_imaginary}.
Even worse, the $\Cl_p^+$ are only known for $p \leqslant 151$.
The $\Cl_p^+$ are \emph{conditionally known} for $157 \leqslant p \leqslant 241$ \cite{Miller}, which we denote by ${}^*$, under the Generalized Riemann Hypothesis for the zeta function of the Hilbert class field of $\Q(\zeta_p+\zeta_p^{-1})$ .
We list these new results of R~Schoof and J~C~Miller:

\begin{table}[htbp]\caption{$\Cl_p^+$: \cite[Theorem 1.1]{Miller} \hfill $\Cl_p^-/2$: \cite[Table 4.4]{Schoof_imaginary}\newline
$H_0(C_2;\Cl_p)$: \hfill This group vanishes${}^*$ for the $46$ primes $p \leqslant 241$ not listed.}\label{table}
\centering
\begin{tabular}{c|c|c|c}
$p$ & $\Cl_p^+$ & $\Cl_p^-/2$ & $H_0(C_2;\Cl_p)$\\
\hline
29 & $0$ & $(2,2,2)$ & $(2,2,2)$ \\
113 & $0$ & $(2,2,2)$ & $(2,2,2)$ \\
163 & $(2,2)^*$ & $(2,2)$ & $4 \leqslant \text{order} \leqslant 16^*$ \\
191 & $(11)^*$ & $0$ & $(11)^*$ \\
197 & $0^*$ & $(2,2,2)$ & $(2,2,2)^*$ \\
229 & $(3)^*$ & $0$ & $(3)^*$ \\
239 & $0^*$ & $(2,2,2)$ & $(2,2,2)^*$
\end{tabular}
\end{table}
\end{rem}
%\noindent K~Uchida showed $\Cl_{163}^+ = \Z/2\oplus \Z/2^*$ using the real cubic field of conductor $163$ \cite{Uchida}.

Theorem~\ref{thm:classification} follows from Theorems~\ref{thm:homotopy} and~\ref{thm:homeo} below.
Consider complex coordinates
\[
S^{2k-1} ~=~ \{ u \in \C^k ~:~ u \cdot \ol{u}  = 1 \}.
\]
For any $q$ coprime to any $d > 1$, there is a linear isometry of $S^{2k-1}$ giving a free $C_d$-action:
\[
\Phi_{d,q} : S^{2k-1} \longra S^{2k-1} ~;~ (u_1,u_2, \ldots,u_k) \longmapsto (\zeta_d^q u_1, \zeta_d u_2, \ldots, \zeta_d u_k).
\]
The quotient manifold $L_{d,q}^{2k-1} := S^{2k-1}/\Phi_{d,q}$ is called \emph{the lens space of type $(d;q,1,\ldots,1)$}.

\begin{rem}\label{rem:classical}
The products of $S^1$ with the classical lens spaces $\Lambda$ of type $(p;q_1,\ldots, q_k)$ and $\Lambda'$ of type $(p;q_1',\ldots, q_k')$ are distinguished in Corollary~\ref{cor:classification}, first by homotopy type in the first factor, and then by homeomorphism type in the other factors, as follows.
First, $\Lambda$ has the homotopy type of $L_{p,q}$, where $q := q_1 \cdots q_k$, and similarly for $\Lambda'$ with $q' := q_1' \cdots q_k'$.
Furthermore, these types are equal if and only if $[q]=[q']$ in the set $Q_p^k$ \cite[(29.4)]{Cohen}.
Now assume $[q]=[q']$, so there exists a homotopy equivalence $f: \Lambda' \longra \Lambda$.

Second, assume $0 = \rho[\Lambda',f] = \rho(\Lambda') - \rho(\Lambda)$, which is independent of the choice of $f$.
Indeed, $\rho$ is an invariant of the $h$-bordism class of $(\Lambda',f)$ \cite[7.5]{AtiyahSinger}.
Then $[\Lambda',f] = [S^1\times \Lambda', \id_{S^1} \times f]$ in $\cS^s_\TOP(S^1 \times \Lambda)$ maps to zero in $\cS^h_\TOP(S^1 \times \Lambda) \iso \Z^{(p-1)/2}$; see Corollary~\ref{cor:h}.
This kernel is identified with the kernel of $\wt{L}^h_{2k}(C_p) \longra \wt{L}^p_{2k}(C_p)$, which is further identified with the following cokernel $\cH(C_p)$ arising in the Ranicki--Rothenberg sequence \cite{Bak_evenL}:
\[
\cH(C_p) ~:=~ \Cok\left( \wh{H}_0(C_2;K_0(\Z C_p,\Q C_p)) \longra \wh{H}_0(C_2;\Cl_p) \right).
\]
So the structure $[\Lambda', f]$ lies in the subquotient $\cH(C_p)$ of the third factor, $H_0(C_2;\Cl_p)$.

Third, assume the given $2$-torsion element $[\Lambda', f]$ of $\cS^h_\TOP(\Lambda)$ vanishes in $H_0(C_2;\Cl_p)$.
Then $f: \Lambda' \longra \Lambda$ is $h$-bordant to the identity map.
In particular, $\Lambda'$ is $h$-cobordant to $\Lambda$.
Therefore, they are isometric \cite[12.12]{Milnor_hcob}; equivalently $\Lambda$ and $\Lambda'$ are homeomorphic.
\end{rem}

For any closed manifold $X$, consider the set $\cM^{h/s}_\TOP(X)$ of closed topological manifolds $M$ homotopy equivalent to $X$ up to homeomorphism.
The calculation of $\cA_\ell$ reduces to $\cM$.

\begin{thm}\label{thm:homotopy}
Let $\ell$ be square-free odd.
Then $\cA_\ell^{2k} = \{(T_\ell)\}$ for all $k>0$ and $\cA_\ell^1 = \{(T_\ell)\}$.
Otherwise, for all $k>1$, passage to orbit spaces induces a bijection
\[\begin{CD}
\cA_\ell^{2k-1} - \{(T_\ell)\} @>>{\approx}> %\arrow[tail, two heads]{r} &
\displaystyle \coprod_{1<d \,|\, \ell} \: \coprod_{[q] \in Q_d^k} \cM^{h/s}_\TOP(S^1 \times L_{d,q}^{2k-1}).
\end{CD}\]
\end{thm}

We calculate these $\cM$ by methods of surgery theory and express them with $K$-theory.

\begin{thm}\label{thm:homeo}
Let $d$ be square-free odd, $q$ coprime to $d$, and $k>1$.
There is a surjection
\[\begin{tikzcd}
\Z^{(d-1)/2} ~\times~ H_0(C_2;\Wh_0(C_d))
\arrow[two heads]{r} & \cM^{h/s}_\TOP(S^1 \times L_{d,q}^{2k-1}).
\end{tikzcd}\]
Any preimage has cardinality dividing $8 \gcd(k,\varphi(d)/2)$, which has bounded growth in $d$.
\end{thm}

Theorem~\ref{thm:homotopy} and Theorem~\ref{thm:homeo} are proven in Section~\ref{sec:homotopy} and Section~\ref{sec:actions}, respectively.
The difficulty in generalizing Theorem~\ref{thm:classification} to all odd $\ell$ comes from the proof of Theorem~\ref{thm:homotopy}.
When $d>1$ is \emph{not square-free,} say $d=p^2$, the groups $NK_1(\Z[C_{p^2}])$ are huge: they are closely related to infinitely generated modules over the Verschiebung algebra of $\F_p[t]$.
Nonetheless, there would be two difficulties in handling elements of $NK_1$ in this paper: topologically there would be a `relaxation' obstruction to making Proposition~\ref{prop:transitive} work, and algebraically there would be a `homothety' obstruction to making Lemma~\ref{lem:eps}(1) work.

%------------------------------------------------------------------------------
\section{Classification of homotopy types}\label{sec:homotopy}

The first stage is the homotopy classification of orbit spaces, then analysis of conjugacy.

\begin{prop}\label{prop:crosscircle}
Let $S^1 \times S^n$ be an $\ell$-fold regular cyclic cover of a topological space $M$, with $n \geqslant 1$ and odd $\ell>1$.
Then $M$ is homotopy equivalent to $S^1 \times S^n$ or $S^1 \times L_{d,q}^n$ with $d | \ell$.
\end{prop}

The degree $\ell$ must be odd, or else the Klein bottle $M=\mathbb{RP}^2\#\mathbb{RP}^2$ is a counterexample.

\begin{proof}
The regular covering map $S^1 \times S^n \longra M$ has degree $\ell > 1$.
Since $\ell$ is odd, the quotient manifold $M$ is oriented.
If $n=1$ then $M$ must be homeomorphic to the torus $S^1 \times S^1$.
If $n=2$ then $M$ must be homotopy equivalent to $S^1 \times S^2$. 
So now assume $n \geqslant 3$.

The covering map $S^1 \times S^n \longra M$ has covering group $C_\ell$.
Write $\Gamma := \pi_1(M)$ for the fundamental group of the quotient space.
The exact sequence of homotopy groups contains
\[\begin{tikzcd}
1 \rar & C_\infty \rar{\iota} & \Gamma \rar{\varphi} & C_\ell \rar & 1.
\end{tikzcd}\]
Write $T \in \Gamma$ for the image under $\iota$ of a generator of $C_\infty$.
Select an element $S \in \Gamma$ such that $S$ maps under $\varphi$ to a generator $s$ of $C_\ell$.
Define a setwise section
\[
\sigma: C_\ell \longra \Gamma ~;~ s^b \longmapsto S^b \text{ for all } 0 \leqslant b < \ell.
\]
In general for a group extension equipped with a setwise section, one has $\Gamma = (\Img\,\iota)(\Img\,\sigma)$.
Then, for each $x \in \Gamma$, we obtain the normal form $x = T^a S^b$ for some $a \in \Z$ and $0 \leqslant b < \ell$.
Note $S^{-1} T S \in \{T,T^{-1}\}$.
If $S^{-1} T S = T^{-1}$ then $S^{-\ell} T S^\ell = T^{(-1)^\ell}$, but $S^\ell \in \Ker\,\varphi = \Img\,\iota$ and $\ell$ is odd, so $T = T^{-1}$, a contradiction.
Hence $TS = ST$.
Therefore $\Gamma$ is abelian.
Hence $\pi_1(M) = \Gamma \iso C_\infty \times C_d$ for some divisor $d$ of $\ell$
(this includes the case of $d=1$).

There is a corresponding infinite cyclic cover $\ol{M}$ with covering translation $t: \ol{M} \longra \ol{M}$.
There is a bundle sequence $\R \longra \Torus(t) \longra M$, with total space the mapping torus of $t$.

Observe that $\ol{M}$ is a $PD_n$-complex, since the $PD_n$-complex $\R \times S^n$ is its universal cover with finite covering group $\pi_1(\ol{M}) = C_d$.
For any $PD_n$-complex $X$ with $n \geqslant 3$ and $\wt{X} \simeq S^n$, Wall showed that the first Postnikov invariant $k_1(X): K(\pi_1 X,1) \to K(\Z,n+1)$ is a generator of the abelian group $H^{n+1}(\pi_1 X;\Z)$ and that the oriented homotopy type of $X$ is uniquely determined by the orbit $[k_1(X)]$ under action of the group $\Out(\pi_1 X)$ \cite[Theorem~4.3]{Wall_PD}.

If $d=1$ then $\ol{M}$ is homotopy equivalent to $S^n$.
Otherwise assume $d>1$.
Recall the cohomology ring $H^*(C_d;\Z)=\Z[\iota]/(d \, \iota)$, where $\iota$ has degree 2; in particular, $K(C_d,1)$ has 2-periodic cohomology.
However $C_d$ acts freely on $\R \times S^n \simeq S^n$, so a standard argument with the Leray--Serre spectral sequence shows that $K(C_d,1)$ has $(n+1)$-periodic cohomology.
Hence $n=2k-1$ for some $k>1$.
Write $q \, \iota^k \in H^{2k}(C_d;\Z) = \Z/d$ for the first Postnikov invariant of $\ol{M}$; we have $\gcd(d,q)=1$.
The lens space $L(d;q,1, \ldots, 1)$ also has first Postnikov invariant $q$, so $\ol{M}$ must be homotopy equivalent to $L_{d,q}^{2k-1} = L(d;q,1, \ldots, 1)$.

In any of these cases of $d$ and $q$, there exist a closed $n$-manifold $L$ and a homotopy equivalence $h: L \longra \ol{M}$.
Select a homotopy inverse $\ol{h}: \ol{M} \longra L$ for $h$; consider the oriented homotopy equivalence $\alpha := \ol{h} \circ t \circ h: L \longra L$.
By cyclic permutation of factors,
\[
\Torus(\alpha) ~\simeq~ \Torus(h \circ \ol{h} \circ t) ~\simeq~ \Torus(t) ~\simeq~ M.
\]
Then on fundamental groups we have $C_d \rtimes_{\alpha_\#} C_\infty \iso C_d \times C_\infty$, where $\alpha_\# \in \Out(C_d)$ is the induced automorphism on $\pi_1(L)$.
Hence $\alpha_\#=\id$.
Therefore $\alpha \simeq \id$ \cite[(29.5A)]{Cohen}.
\end{proof}

The linking form on the $(k-1)$-st homology group of the infinite cyclic cover $\ol{M}$ is the $1 \times 1$ matrix $[q/p] \in \Q/\Z$ \cite[\S77: p290]{ST}, which recovers the Postnikov invariant $q \, \iota^k$.

In the sequel, we shall fix $k>1$ and consider the latter, closed $2k$-dimensional manifold
\[
X_{d,q} ~:=~ S^1 \times L_{d,q}^{2k-1}.
\]
The following definition generalizes the homeomorphism of Jahren--Kwasik~\cite[\S4]{JK}.
Write $t$ and $s$ for the usual generators of $C_\infty$ and $C_d$, respectively.
Note $(t^k,s^j) \longmapsto (t^k,s^{k+j})$ in $\Aut(C_\infty \times C_d)$ is induced by the well-defined self-homeomorphism (like a Dehn twist):
\begin{equation}\label{eqn:epsh}
\eps : X_{d,q} \longra X_{d,q} ~;~ (z,[u_1 : u_2 : \ldots : u_k]) \longmapsto (z,[z^{q/d} u_1 : z^{1/d} u_2 : \ldots : z^{1/d} u_k]).
\end{equation}
This is multiplication by the path $[0,2\pi] \longra GL_k(\C) ~;~ \theta \longmapsto \mathrm{diag}(e^{\theta i q/d}, e^{\theta i/d}, \ldots, e^{\theta i/d})$.

\begin{prop}\label{prop:transitive}
Let $f: M \longra X_{d,q}$ be a homotopy equivalence with $M$ a closed manifold.
There exists $\delta \in \Homeo(M)$ satisfying a homotopy commutative diagram
\[\begin{tikzcd}
M \rar{f} \dar[swap]{\delta} & X_{d,q} \dar{\eps^2}\\
M \rar{f} & X_{d,q}.
\end{tikzcd}\]
\end{prop}

Later, in Section~\ref{sec:eps}, we prove Proposition~\ref{prop:transitive} based on surgery-theoretic calculations.

Notice that $\pi_1(X_{d,q}) = C_\infty \times C_d$ does \emph{not} have a \emph{unique} infinite cyclic subgroup $Z$ of index $d$, rather there are exactly $d$ such subgroups (generated by $ts^r$ with $0 \leqslant r < d$).
Although each $Z$ is normal, none is characteristic: $\Aut(C_\infty \times C_d)$ acts transitively on them.

\begin{cor}\label{cor:transitive}
Let $M$ be a closed manifold in the homotopy type of $X_{d,q}$.
Let $Z$ and $Z'$ be infinite cyclic subgroups of index $d$ in $\pi_1(M)$.
Then $\delta'_\#(Z) = Z'$ for some $\delta' \in \Homeo(M)$.
\end{cor}

\begin{proof}
Select a homotopy equivalence $f: M \longra X_{d,q}$.
There are integers $a$ and $b$ such that $f_\#(Z)$ and $f_\#(Z')$ are generated respectively by $ts^a$ and $ts^b$ in $\pi_1(X_{d,q})$.
By Proposition~\ref{prop:transitive}, there is $\delta \in \Homeo(M)$ with $f \circ \delta \simeq \eps^2 \circ f$.
Define $\delta' := \delta^{(b-a)(1-d)/2} \in \Homeo(M)$.
Note
\[
\delta'_\#(f^{-1}_\#(ts^a)) ~=~ f^{-1}_\#(\eps_\#^{(b-a)(1-d)}(ts^a)) ~=~ f^{-1}_\#(ts^{(b-a)(1-d)}s^a) ~=~ f^{-1}_\#(ts^b).
\qedhere
\]
\end{proof}

\begin{proof}[Proof of Theorem~\ref{thm:homotopy}]
Conjugate subgroups of $\Homeo(S^1 \times S^n)$ give homeomorphic orbit spaces.
Then, by Proposition~\ref{prop:crosscircle}, we can define a function $\Phi$ given by homeomorphism classes of homotopy types of orbit spaces:
\[\begin{CD}
\Phi ~:~ \cA_\ell^n @>>> \displaystyle \cM^{h/s}_\TOP(S^1 \times S^n) ~\sqcup~
\begin{cases}
\varnothing & \text{if } n=1 \text{ or } n=2k\\
\coprod_{1<d|\ell} \coprod_{[q] \in Q_d^k} \cM^{h/s}_\TOP(X_{d,q}) & \text{if } n=2k-1 \geqslant 3.
\end{cases}
\end{CD}\]

Note $\Phi\{(T_\ell)\} = \{[S^1\times S^n]\} = \cM^{h/s}_\TOP(S^1 \times S^n)$, where the latter equality follows from: classification of surfaces if $n=1$, Thurston's Geometrization Conjecture if $n=2$ (see \cite{Anderson}), and the topological surgery sequence \cite{KS} if $n \geqslant 3$ (use \cite{FQ_book} if $n=3$).

First suppose $n=1$.
Then, as noted above, $\Phi$ is constant hence surjective.
(Since $\ell$ is odd, only the torus $S^1 \times S^1$ has $\ell$-fold cover $S^1 \times S^1$.
That is, $\Phi(\cA_\ell^1) = \{[S^1 \times S^1]\}$.)

Let $(C) \in \cA_\ell^1$.
There exists a choice of homeomorphism $h: (S^1 \times S^1)/C \longra S^1 \times S^1$.
Under the quotient map $S^1 \times S^1 \longra (S^1 \times S^1)/C$ composed with $h$, the image of the fundamental group of $S^1 \times S^1$ is a subgroup $Z$ of index $\ell$ in $\pi_1(S^1 \times S^1) = \Z \times \Z$.
There exists a nontrivial homomorphism $\phi: \Z \times \Z \longra \Z/\ell$ such that $Z = \Ker(\phi)$.
Write $a := \phi(1,0)$ and $b := \phi(0,1)$.
Post-composition with an automorphism of $\Z/\ell$ preserves the kernel $Z$, so we may assume that either $a=1$ or $(a,b)=(0,1)$.
If $a=1$ then define $A := \left[\begin{smallmatrix}1 & 0\\ -b & 1\end{smallmatrix}\right]$.
If $(a,b)=(0,1)$ then define $A := \left[\begin{smallmatrix}0 & 1\\ 1 & 0\end{smallmatrix}\right]$.
In any case, the unimodular matrix $A \in GL_2(\Z/\ell)$ carries $(a,b)$ to $(1,0)$. %\footnote{More generally: the subgroups of index $d$ in $\Z^m$ correspond to the points of the projective space $\mathbb{P}^m(\Z/d)$. This correspondence is $GL_m(\Z/d)$-equivariant up to a transpose, and  $GL_m(\Z/d)$ acts transitively on $\mathbb{P}^m(\Z/d)$.}
Observe $(1,0)$ corresponds to the index $\ell$ subgroup $\ell \Z \times \Z$.
There is $\delta' \in \Homeo(S^1 \times S^1)$ inducing $A$ on fundamental group.
Write $h' := \delta' \circ h$.
Then, by the lifting property of covering spaces, there exists a commutative diagram
\[\begin{tikzcd}
S^1 \times S^1 \rar[dashed]{\wh{h'}} \dar{/C} & S^1 \times S^1 \dar{/T_\ell}\\
(S^1 \times S^1)/C \rar{h'} & S^1 \times S^1.
\end{tikzcd}\]
The element $\wh{h'} \in \Homeo(S^1 \times S^1)$ conjugates $T_\ell$ into $C$.
Therefore $\Phi$ is injective.

Now suppose $n>1$ and that the orbit space of $(C) \in \cA_\ell^n$ is homeomorphic to $S^1 \times S^n$, say by a homeomorphism $h$.
Since $\pi_1(S^1 \times S^n) = C_\infty$ has a unique subgroup of index $\ell$, by the lifting property of covering spaces, there exists a commutative diagram
\[\begin{tikzcd}
S^1 \times S^n \rar{\wh{h}} \dar{/C} & S^1 \times S^n \dar{/T_\ell}\\
(S^1 \times S^n)/C \rar{h} & S^1 \times S^n.
\end{tikzcd}\]
In other words, there is $\wh{h} \in \Homeo(S^1 \times S^n)$ that conjugates $T_\ell$ into $C$.
Thus $\Phi$ restricts to
\[\begin{CD}
\Phi ~:~ \cA_\ell^n - \{(T_\ell)\} @>>>\begin{cases}
\varnothing & \text{if } n=1 \text{ or } n=2k\\
\coprod_{1<d|\ell} \coprod_{[q] \in Q_d^k} \cM^{h/s}_\TOP(X_{d,q}) & \text{if } n=2k-1 \geqslant 3.
\end{cases}\end{CD}\]

Next, we show that $\Phi$ is surjective if $n=2k-1 \geqslant 3$.
Let $M$ be a closed manifold in the homotopy type of some example $X_{d,q}$, say by a homotopy equivalence $f$.
There is a pullback diagram of covering spaces
\[\begin{tikzcd}
\wh{M} \rar{\wh{f}} \dar & S^1 \times S^n \dar{/T_{d,q}} \\
M \rar{f} & X_{d,q}.
\end{tikzcd}\]
Let $T \neq \id$ be a covering transformation of $\wh{M}$.
Since $\cM^{h/s}_\TOP(S^1 \times S^n) = \{[S^1\times S^n]\}$, there is a homeomorphism $h: \wh{M} \longra S^1 \times S^n$.
Then $T_M := h \circ T \circ h^{-1}$ is an element of $\Homeo(S^1 \times S^n)$ of order $d$ without fixed points.
Hence $M = \Phi(T_M)$ and $\Phi$ is surjective.

Finally, we show that $\Phi$ is injective if $n=2k-1 \geqslant 3$.
Let $(C), (C') \in \cA_\ell^n$ have orbit spaces $M, M'$ in the homotopy type of some example $X_{d,q}$.
Suppose there is a homeomorphism $h: M' \longra M$.
Write $\Pi := \pi_1(S^1 \times S^n)$.
Consider the lifting problem
\[\begin{tikzcd}
S^1 \times S^n \rar[dashed] \arrow{d}[swap]{p'} & S^1 \times S^n \dar{p}\\
M' \rar{h} & M.
\end{tikzcd}\]
By Corollary~\ref{cor:transitive}, there exists $\delta' \in \Homeo(M)$ such that $\delta'_\#((h \circ p')_\#(\Pi)) = p_\#(\Pi)$.
Note $h' := \delta' \circ h: M' \longra M$ satisfies $(h' \circ p')_\#(\Pi) = p_\#(\Pi)$.
Then, by the lifting property, there is $\wh{h'} \in \Homeo(S^1 \times S^n)$ covering $h'$ that conjugates $C'$ to $C$.
Therefore $\Phi$ is injective.
\end{proof}

\noindent See \cite{Thatcher} for the homotopy types of free $C_p$-actions on products of 1-connected spheres.

%------------------------------------------------------------------------------
\section{Classification of $h$-cobordism types}\label{sec:hcobordism}

For the second stage, consider the subgroup $SI(X)$ of $\Wh_1(\pi_1 X)$ consisting of the Whitehead torsions of \emph{strongly inertial $h$-cobordisms}, that is, the torsion $\tau(W \thra X)$ of any $h$-cobordism $(W;X,X')$ such that the map $X' \hookra W \thra X$ is homotopic to a homeomorphism.

\begin{thm}\label{thm:SI}
Let $M$ and $X$ be closed connected topological manifolds of dimension $n \geqslant 4$.
If $n=4$ then assume $\pi_1 X$ is good in the sense of Freedman--Quinn \cite{FQ_book}.
If $M$ is homotopy equivalent to $X$, then $SI(M) \iso SI(X)$ as subgroups of $\Wh_1(\pi_1 M) \iso \Wh_1(\pi_1 X)$.
\end{thm}

This theorem is an affirmative answer to a question raised by Jahren--Kwasik \cite[\S7]{JK_hcob}.
Later, in Section~\ref{sec:homeomorphism}, we shall develop the techniques needed to prove this theorem.  

Next, for any compact manifold $X$, write $\cS^{h/s}_\TOP(X)$ for the set
%\marginpar{\tiny Shmuel: this is an abelian group, using periodicity of [Cappell--Weinberger]}
of pairs $(M,f)$, where $M$ is a compact topological manifold and $f: M \longra X$ is a homotopy equivalence that restricts to a homeomorphism $\bdry f: \bdry M \longra \bdry X$, taken up to $s$-bordism relative to $\bdry X$.
Assuming that the $s$-cobordism theorem applies, then $[M,f]=[M',f']$ if and only if $f'$ is homotopic to $f \circ h$ relative to $\bdry X$ for some homeomorphism $h: M' \longra M$.
Then observe
\[
\cM_\TOP^{h/s}(X) ~=~ \hMod(X) ~\backslash~ \cS^{h/s}_\TOP(X). 
\]
Here $\cS^{h/s}_\TOP(X)$ has a canonical left action by the group $\hMod(X)$, which consists of homotopy equivalences $X \to X$ restricting to the identity on $\bdry X$, taken up to homotopy $\rel\,\bdry X$.

The first step in proving Theorem~\ref{thm:homeo} is an observation of Jahren--Kwasik \cite[\S3]{JK_hcob}.
In the definition of $\cS^{h/s}_\TOP(X)$, weaken the equivalence relation ``$s$-bordism''  to ``$h$-bordism.''
Then the resulting set $\cS^h_\TOP(X)$ has the structure of an abelian group, according to Ranicki \cite{Ranicki_TopMan}.
Hence $\cS^h_\TOP(X)$ is more calculable; it also has a left setwise action of $\hMod(X)$.

\begin{prop}[Jahren--Kwasik]\label{prop:hs}
Let $X$ be a closed connected topological manifold of dimension $n \geqslant 4$.
If $n=4$ then assume $\pi_1 X$ is good in the sense of Freedman--Quinn~\cite{FQ_book}.
The set $\cS^{h/s}_\TOP(X)$ has a canonical right action of the Whitehead group $\Wh_1(\pi_1 X)$, so that
\[
\cS^{h}_\TOP(X) ~=~ \cS^{h/s}_\TOP(X) ~/~ \Wh_1(\pi_1 X).
\]
The isotropy group of any element $[M,f]$ in $\cS^{h/s}_\TOP(X)$ is the subgroup $f_*SI(M)$.
The forgetful map $\cS^{h/s}_\TOP(X) \longra \cS^{h}_\TOP(X)$ is equivariant with respect to the left action of $\hMod(X)$.
\end{prop}

Only the isotropy group of $[M,f]=[X,\id]$ is proven in \cite[\S3]{JK_hcob}; we prove the others.

\begin{proof}
Recall the canonical left action.
Let $\gamma \in \hMod(X)$ and $[M,f] \in \cS^{h/s}_\TOP(X)$.
Define
\[
\gamma \cdot [M,f] ~:=~ [M,\gamma\circ f].
\]
The left action on $\cS^h_\TOP(X)$ has the same formula, so the forgetful map is equivariant.

Next, recall the canonical right action.
Let $[M,f] \in \cS^{h/s}_\TOP(X)$ and $\alpha \in \Wh_1(\pi_1 X)$.
By realization, there is an $h$-cobordism $(W;M,M')$ with torsion $\tau(W \thra M) = f_*^{-1}(\alpha)$.
Define
\[
[M,f] \cdot \alpha ~:=~ [M',f \circ (M \twoheadleftarrow W \leftarrow M')].
\]
This is well-defined in $\cS^{h/s}_\TOP(X)$ since $(W;M,M')$ is unique up to homeomorphism $\rel M$.
Thus the forgetful map induces a function $\cS^{h/s}_\TOP(X) / \Wh_1(\pi_1 X) \longra \cS^h_\TOP(X)$, a bijection.

Finally, we determine isotropy groups of the right action.
Clearly $f_*SI(M)$ fixes $[M,f]$.
Suppose $[M,f] \cdot \alpha = [M,f]$.
Abbreviate the homotopy equivalence $g_\alpha := (M' \hookra W \thra M)$.
Then $f \circ g_\alpha$ is $s$-bordant to $f$.
By the $s$-cobordism theorem, there exists a homeomorphism $h: M' \longra M$ such that $f \circ g_\alpha$ is homotopic to $f \circ h$.
By post-composition with a homotopy inverse $\ol{f}: X \longra M$ of $f$, we have $g_\alpha$ is homotopic to $h$.
Therefore $f_*^{-1}(\alpha) \in SI(M)$.
\end{proof}

In general, when $X=S^1 \times Y$, the Ranicki--Shaneson decomposition for $L^h$-groups \cite{Ranicki_AlgL2} induces a corresponding decomposition for the $h$-structure groups \cite[C1]{Ranicki_TopMan}.

\begin{prop}[Ranicki]\label{prop:Shaneson}
Let $Y$ be a topological space, and let $m$ be an integer.
There is a functorial isomorphism of algebraic structure groups:
\[
\cS^h_m(S^1 \times Y) ~\iso~ \cS^h_m(Y) ~\oplus~ \cS^p_{m-1}(Y).
\]
\end{prop}

Further suppose that $Y$ is a closed connected topological manifold of dimension $n-1$.
The total surgery obstruction of Ranicki \cite[Theorem~18.5]{Ranicki_TopMan} gives the identifications
\[\begin{tikzcd}
\cS^h_\TOP(S^1 \times Y) \rar{s}[swap]{\approx} & \cS^h_{n+1}(S^1 \times Y)
\end{tikzcd}
\text{and}
\begin{tikzcd}
\cS^h_\TOP(I \times Y) \rar{s}[swap]{\approx} & \cS^h_{n+1}(Y).
\end{tikzcd}
\]
Since $s$ exists for all dimensions $n$, by the Five Lemma applied to the 4-dimensional surgery sequence \cite[\S 11.3]{FQ_book}, we also have these bijections when $n=4$ and $\pi_1 Y$ is finite.

The next two lemmas determine certain $\cS_*(Y)$ when $Y$ is a lens space of odd order.

\begin{lem}\label{lem:crossI}
Let $d>1$ be odd, select $q$ coprime to $d$, and let $k > 1$.
Then $\cS^{s,h}_{2k+1}(L_{d,q}^{2k-1}) = 0$.
\end{lem}

\begin{proof}
Write $L^n := L_{d,q}^{2k-1}$.
Consider the $s$- or $h$-algebraic surgery exact sequence \cite{Ranicki_TopMan}:
\[\begin{tikzcd}
L_{2k+1}^{s,h}(C_d) \rar & \cS^{s,h}_{2k+1}(L^n) \rar & H_{2k}(L^n;\bL\!\gens{1}) \rar{\sigma_{2k}^{s,h}} & L_{2k}^{s,h}(C_d).
\end{tikzcd}\]
First, since $d$ is odd, $L_{2k+1}^{s,h}(C_d)=0$ by Bak's vanishing result \cite{Bak_oddL}.
Next, we apply the Atiyah--Hirzebruch spectral sequence to the homological version of the normal invariants:
\[
E^2_{i,j} = H_i(L^n; L\!\gens{1}_j) ~\Longrightarrow~ H_{i+j}(L^n;\bL\!\gens{1}).
\]
The coefficient group $L\!\gens{1}_j$ vanishes for $j \leqslant 0$ or $j$ odd.
Otherwise, it either is $\Z$ if $j \equiv 0 \!\pmod{4}$ or is $\Z/2$ if $j \equiv 2 \!\pmod{4}$.
Note that $\wt{H}_{even}(L^n;\Z) = 0$ and, since $d$ is odd, that $\wt{H}_{even}(L^n;\Z/2) = 0$.
Thus the diagonal entries $i+j=even$ are zero except along $i=0$.
Also note that $H_{odd}(L^n;\Z) \in \{0, \Z/d, \Z\}$ and, since $d$ is odd, that $H_{odd}(L^n;\Z/2)=0$.
Therefore, since the image of an odd-order group in either $\Z$ or $\Z/2$ is zero, in summary we obtain:
\begin{equation}\label{eqn:evenhomology}
H_{2k}(L^n;\bL\!\gens{1}) = E^\infty_{0,2k} = E^2_{0,2k} = L\!\gens{1}_{2k} = L_{2k}(1).
\end{equation}
Thus the assembly map is injective, $\sigma_{2k}^{s,h}: L_{2k}(1) \longra L_{2k}^{s,h}(C_d)$.
Hence $\cS^{s,h}_{2k+1}(L^n)=0$.
\end{proof}

\begin{lem}\label{lem:projective}
Let $d>1$ be odd, select $q$ coprime to $d$, and let $k > 1$.
Then $\cS^p_{2k}(L_{d,q}^{2k-1})$ is free abelian of rank $(d-1)/2$.
Moreover, $\wt{L}_{2k}^p(C_d) \longra \cS^p_{2k}(L_{d,q}^{2k-1})$ is injective with finite index.
\end{lem}

\begin{proof}
Write $L^n := L_{d,q}^{2k-1}$.
Consider the $p$-algebraic surgery exact sequence \cite{Ranicki_TopMan}:
\[\begin{tikzcd}
H_{2k}(L^n;\bL\!\gens{1}) \rar{\sigma_{2k}^p} & L_{2k}^p(C_d) \rar & \cS^p_{2k}(L^n) \rar & H_{2k-1}(L^n;\bL\!\gens{1}) \rar{\sigma_{2k-1}^p} & L_{2k-1}^p(C_d).
\end{tikzcd}\]
From the proof of Lemma~\ref{lem:crossI}, the edge map $L_{2k}(1) \longra H_{2k}(L^n;\bL\!\gens{1})$ is an isomorphism, so $\sigma_{2k}^p$ is split injective.
Also $\sigma_{2k-1}^p$ is zero, since it factors through $L_{2k-1}^h(C_d)=0$ above.
So we obtain an exact sequence of abelian groups:
\[\begin{tikzcd}
0 \rar & \wt{L}_{2k}^p(C_d) \rar & \cS^p_{2k}(L^n) \rar & H_{2k-1}(L^n;\bL\!\gens{1}) \rar & 0.
\end{tikzcd}\]

Since $\R C_d = \R \times \prod^{(d-1)/2} \C$ as rings, the reduced $L$-group $\wt{L}_{2k}^p(C_d)$ is free abelian of rank $(d-1)/2$, and it is detected by the projective multi-signature \cite{Bak_evenL}.
From the same Atiyah--Hirzebruch spectral sequence as in the proof of Lemma~\ref{lem:crossI}, since $d$ is odd, note:
\[
E^2_{i,j} = H_i(L^n; L\!\gens{1}_j) = \begin{cases}
\Z & \text{if } i=2k-1, \text{ and } 4 \text{ divides } j>0\\
\Z/d & \text{if } 0<i<2k-1 \text{ odd, } 4 \text{ divides } j>0\\
0 & \text{otherwise}
\end{cases}
~\Longrightarrow~ H_{i+j}(L^n;\bL\!\gens{1}).
\]
Then each $E^\infty_{i,j}$ is either zero or $\Z/\delta$ with $\delta \,|\, d$.
Thus it follows that $H_{2k-1}(L^n;\bL\!\gens{1})$ is a finite abelian group of odd order.\footnote{A more detailed analysis can show furthermore that $H_{2k-1}(L^n;\bL\!\gens{1}) \to H_{2k-1}(L^n;ko[\frac{1}{2}])$ is an isomorphism.}
Therefore it remains to show that $\cS^p_{2k}(L^n)$ has no odd torsion.

The function $\cS^s_\TOP(L^n) \longra \Q \otimes_\Z R_{\wh{G}}^{(-1)^k}$, defined by the difference of $\rho$-invariants, was shown by Wall to be injective \cite[Theorem~14E.7]{Wall_book}.\footnote{For the case of $k=2$, we use the \emph{simple homology structure set} of the 3-dimensional lens space $L^3 = L(d,q)$.}
Later, Macko--Wegner promoted this function to a homomorphism of abelian groups and reproved its injectivity \cite[Theorem~5.2]{MW}.
So $\cS^s_\TOP(L^n)$ is free abelian.
By the Ranicki--Rothenberg exact sequences \cite[p327]{Ranicki_TopMan}, $\cS^s_{2k}(L^n) \longra \cS^h_{2k}(L^n)$ and $\cS^h_{2k}(L^n) \longra \cS^p_{2k}(L^n)$ have kernels and cokernels of exponent two.
Hence $\cS^p_{2k}(L^n)$ has no odd torsion; it is free abelian of rank $(d-1)/2$.
\end{proof}

\begin{cor}\label{cor:h}
Let $d>1$ be odd, select $q$ coprime to $d$, and let $k > 1$.
Then the group $\cS^h_\TOP(S^1 \times L_{d,q}^{2k-1})$ is free abelian of rank $(d-1)/2$.
Moreover, the component homomorphism $\wt{L}_{2k}^p(C_d) \longra L_{2k+1}^h(\pi_1 X_{d,q}) \longra \cS^h_\TOP(X_{d,q})$ of Wall realization is injective with finite index.
\end{cor}

\begin{proof}
This is immediate from Proposition~\ref{prop:Shaneson}, Lemma~\ref{lem:crossI}, and Lemma~\ref{lem:projective}.
\end{proof}

\section{Application to the `Dehn twist' homeomorphism}\label{sec:eps}

Fix $n=2k-1\geqslant 3$.
Recall the self-homeomorphism $\eps$ of $X_{d,q} = S^1 \times L^n_{d,q}$ in Equation~\eqref{eqn:epsh}.

\begin{lem}\label{lem:eps}
Let $d>1$ be an odd integer, and select $q$ coprime to $d$.
\begin{enumerate}
\item
The self-map $\eps$ induces the identity map on $\Wh_1(\pi_1 X_{d,q})$ if $d$ is square-free.

\item
The self-map $\eps$ induces the identity map on $\cS_\TOP^h(X_{d,q})$.

\item
The self-map $\eps^2$ induces the identity map on $\cS_\TOP^s(X_{d,q})$.
\end{enumerate}
\end{lem}

The $d=2$ case for Part~(2) was a key technical assertion of Jahren--Kwasik \cite[\S4]{JK}.

\begin{rem}\label{rem:Milnor_mistake}
Milnor \cite[1.6]{Milnor_hcob} falsely claimed that $SK_1(\Z G)=0$ for all finite abelian groups $G$; when $G=C_{p^2} \times C_{p^2}$ this $SK_1$-group is isomorphic to $(\Z/p)^{p-1}$  \cite[9.8(ii)]{Oliver_book}.
However it holds for all finite cyclic groups $G=C_n$ by Bass--Milnor--Serre \cite[XI:7.3]{Bass_book}.
Therefore the determinant map $K_1(\Z C_n) \longra (\Z C_n)^\times$ is an isomorphism.
By a theorem of Higman \cite[XI:7.1a]{Bass_book}, the torsion subgroup of $(\Z C_n)^\times$ is $\pm C_n$.
Hence $\Wh_1(C_n)$ is free abelian.
Consequently, the proof of \cite[Lemma~6.7]{Milnor_hcob} still holds in this case, so that the group-ring involution $(g \longmapsto g^{-1})$ induces the identity on the Whitehead group $\Wh_1(C_n)$.
\end{rem}

\begin{proof}[Proof of Lemma~\ref{lem:eps}(1)]
On the fundamental group $\pi_1(X_{d,q}) = C_\infty \times C_d$, recall that $\eps$ induces $(t^k,s^j) \longmapsto (t^k,s^{k+j})$; it is the identity on the subgroup $C_d$, which is generated by $s$.
Then, by Proposition~\ref{prop:Wh}(1), we obtain a commutative diagram whose rows are split exact:
\[\begin{tikzcd}
0 \rar & \Wh_1(C_d) \rar \dar{\id} & \Wh_1(\pi_1 X_{d,q}) \rar{\bdry} \dar{\eps_*} & \Wh_0(C_d) \rar \dar{\eps_*} & 0\\
0 \rar & \Wh_1(C_d) \rar & \Wh_1(\pi_1 X_{d,q}) \rar{\bdry} & \Wh_0(C_d) \rar & 0.
\end{tikzcd}\]
Here $R := \Z[C_d]$, and $\eps: R[t,t^{-1}] \longra R[t,t^{-1}]$ restricts to ring maps $\eps: R[t^{\pm 1}] \longra R[t^{\pm 1}]$.

Now, the splitting of the epimorphism $\bdry$ of Bass--Heller--Swan \cite[XII:7.4]{Bass_book} is
\[
h: \Wh_0(C_d) \longra \Wh_1(\pi_1 X_{d,q}) ~;~ [P] \longmapsto [t: P[t,t^{-1}] \to P[t,t^{-1}]].
\]
Here $P$ is a finitely generated projective $R$-module.
Then note
\[
\eps_*[P]
~=~ (\eps_* \circ \bdry \circ h)[P]
~=~ (\bdry \circ \eps_*) [t: P[t,t^{-1}] \to P[t,t^{-1}]].
\]
Since $\eps(t)=st$, and since $\eps(s)=s$ implies $(R \hookra R[t,t^{-1}] \xra{\eps} R[t,t^{-1}]) = (R \hookra R[t,t^{-1}])$,
\[
\eps_*[t: P[t,t^{-1}] \to P[t,t^{-1}]]
~=~ [s t: P[t,t^{-1}] \to P[t,t^{-1}]].
\]
Recall \cite[IX:6.3]{Bass_book} the map $\bdry$ in the localization sequence for $R[t] \to R[t,t^{-1}]$:
\[
\eps_*[P]
~=~ \bdry[s t: P[t,t^{-1}] \to P[t,t^{-1}]]
~=~ [\Cok(s t: P[t] \to P[t])]
~=~ [P].
\]
So $\eps_*=\id$ on $\Wh_0(C_d)$.
Moreover, in $\Wh_1(\pi_1 X_{d,q})$ note
\begin{eqnarray*}
\eps_*(h[P]) - h[P] &=& [s: P \to P] \in \Wh_1(C_d)\\
d \cdot [s: P \to P] &=& [s^d=1: P \to P] = 0.
\end{eqnarray*}
Thus, since $\Wh_1(C_d)$ is torsion-free by Remark~\ref{rem:Milnor_mistake}, we obtain
\[
\eps_* ~=~ \begin{pmatrix}\id & 0 \\ 0 & \id \end{pmatrix} ~\text{ on }~ \Wh_1(\pi_1 X_{d,q}) = \Wh_1(C_d) \oplus \Wh_0(C_d).
\]
Therefore $\eps$ induces the identity automorphism on $\Wh_1(\pi_1 X_{d,q})$.
\end{proof}

\begin{proof}[Proof of Lemma~\ref{lem:eps}(2)]
By Corollary~\ref{cor:h}, it suffices to show that $\eps_* = \id$ on $L^p_{2k}(C_d)$.
Its definition is $\eps_* := B \circ \eps_* \circ \ol{B}$, in terms of the induced automorphism $\eps_*: L^h_{2k+1}(C_\infty\times C_d) \longra L^h_{2k+1}(C_\infty\times C_d)$, the epimorphism $B: L^h_{2k+1}(C_\infty\times C_d) \longra L^p_{2k}(C_d)$, and its algebraic splitting $\ol{B}: L^p_{2k}(C_d) \longra L^h_{2k+1}(C_\infty\times C_d)$ of Ranicki \cite[Theorem~1.1]{Ranicki_AlgL2}.
Then, heavily using Ranicki's notation and slightly modifying his proof of splitness \cite[p.~134]{Ranicki_AlgL3}, note:
\begin{eqnarray*}
\eps_*[Q,\varphi] &=& (B \circ \eps_* \circ \ol{B}) [Q,\varphi]\\
&=& B\left[ (Q_t \oplus Q_t, \varphi \oplus -\varphi) ~\oplus~ \cH_\pm(-Q_t);\; \Delta_{(Q_t,\varphi)} ~\oplus~ -Q_t,\; \left(\begin{smallmatrix}1 & 0\\ 0 & st\end{smallmatrix}\right) \Delta_{(Q_t, \varphi)} ~\oplus~ -Q_t \right]\\
&=& \left[ B_1^+\left( \Delta_{(Q,\varphi)} \oplus \Delta^*_{(Q^*,\psi)}, \left(\begin{smallmatrix}1 & 0\\ 0 & st\end{smallmatrix}\right) (\Delta_{(Q,\varphi)} \oplus \Delta^*_{(Q^*,\psi)}) \right),\; \varphi \oplus -\varphi \right] ~\oplus~ [\cH_\pm(-Q)]\\
&=& \left[ B_1^+(Q \oplus Q, Q \oplus st Q),\; \varphi \oplus -\varphi \right] ~\oplus~ [\cH_\pm(-Q)]\\
&=& [Q,\varphi] ~\in~ L^p_{2k}(C_d).
\end{eqnarray*}
Here, the equivalence classes are of various quadratic forms and formations.
We only used that the $\Z[C_d]$-algebra map $\eps_\#: \Z[C_d][t,t^{-1}] \longra \Z[C_d][t,t^{-1}]$ is graded of degree 0.
\end{proof}

\begin{proof}[Proof of Lemma~\ref{lem:eps}(3)]
Observe $\eps_*$ respects the Ranicki--Rothenberg exact sequence
\[\begin{tikzcd}
\wh{H}^{n+3}(C_2;\Wh_1 X_{d,q}) \rar & \cS_\TOP^s(X_{d,q}) \rar & \cS_\TOP^h(X_{d,q}) \rar & \wh{H}^{n+2}(C_2;\Wh_1 X_{d,q}).
\end{tikzcd}\]
In particular, by Corollary~\ref{cor:h}, this restricts to an exact sequence
\begin{equation}\label{eqn:forget}
\begin{tikzcd}
0 \rar & H \rar & \cS_\TOP^s(X_{d,q}) \rar & K \rar & 0
\end{tikzcd}
\end{equation}
with $H$ finite abelian and $K$ free abelian.
By Lemma~\ref{lem:eps}(1,2), $\eps_* = \id$ on $H$ and $K$.
Hence
\[
\eps_*
~=~ \begin{pmatrix}\id_{H} & \nu\\ 0 & \id_{\iota K} \end{pmatrix}
~\text{ on }~ \cS_\TOP^s(X_{d,q}) = H \oplus \iota K,
\]
where $\nu: K \longra H$ is a component of $\eps_*$ and $\iota: K \longra \cS_\TOP^s(X_{d,q})$ is a choice of right-inverse of $\cS^s_\TOP(X_{d,q}) \longra K$.
Since $2H=0$, note $2 \nu = 0$.
Hence $\eps_*^2=\id$ on $\cS^s_\TOP(X_{d,q})$.
\end{proof}

We show that the homotopy-theoretic order of $\eps$ divides $2d^2$; see more in Proof~\ref{cor:hMod_lens}.

\begin{lem}\label{lem:eps_power}
The homeomorphism $\eps^{2d^2}$ is homotopic to the identity on $X_{d,q} = S^1 \times L^n$.
\end{lem}

\begin{proof}
Observe that the $d$-th power of $\eps$ induces the identity on fundamental group:
\[
\eps^d ~:~ S^1 \times L^n \longra S^1 \times L^n ~;~ (z,[u_1: u_2: \ldots : u_k]) \longmapsto (z,[z^q u_1: z u_2: \ldots : z u_k]).
\]
Each $1 \leqslant j \leqslant k$ has an isotopy of diffeomorphisms that lifts the generator of $\pi_1(SO_3) = C_2$:
\begin{equation}\label{eqn:rhos}
\rho_j ~:~ S^1 \times L^n \longra S^1 \times L^n ~;~ (z,[u_1: \ldots : u_j: \ldots: u_k]) \longmapsto (z,[u_1: \ldots : z u_j: \ldots : u_k]).
\end{equation}
In the proof of \cite[Proposition 3.1]{HJ}, Hsiang--Jahren showed that each homotopy class $[\rho_j]$ has order $2d$ in the group $\pi_1(\Map\,L^n,\id)$.
Since $S^1$ is a co-$H$-space and $\Diff\,L^n$ is an $H$-space, the two multiplications on $\pi_1(\Diff\,L^n,\id)$ are equal (and abelian), so
\begin{equation}\label{eqn:eps_power}
[\eps^d] ~=~ [\rho_1^q \circ \rho_2 \circ \cdots \circ \rho_k] ~=~ [\rho_1]^q * [\rho_2] * \cdots * [\rho_k] ~\in~ \pi_1(\Diff\,L^n,\id).
\end{equation}
Therefore $[\eps^{2d^2}] = [\eps^d]^{2d} = [\rho_1]^{2dq} [\rho_2]^{2d} \cdots [\rho_k]^{2d} = 1$ in $\pi_1(\Map\,L^n,\id)$.
\end{proof}

Structure sets quantify homeomorphism types within a homotopy type, so we can start:

\begin{proof}[Proof of Proposition~\ref{prop:transitive}]
Consider the homotopy equivalence $\alpha := \ol{f} \circ \eps^2 \circ f: M \longra M$, where $\ol{f}$ denotes a homotopy inverse for $f$.
By the composition formula for Whitehead torsion \cite[Lemma~7.8]{Milnor_hcob}, by topological invariance \cite{Chapman}, and by Lemma~\ref{lem:eps}(1),
\[
\tau(\alpha) ~=~ \tau(\ol{f}) + \ol{f}_* (\tau(\eps^2) + \eps^2_* \tau(f))
~=~ -f^{-1}_*\tau(f) + f^{-1}_* (0 + \tau(f)) ~=~ 0 ~\in~ \Wh_1(\pi_1 M).
\]
That is, $\alpha$ is a simple homotopy equivalence, hence it defines an element $[M,\alpha] \in \cS^s_\TOP(M)$.

On the other hand, by Lemma~\ref{lem:eps}(3) and Lemma~\ref{lem:eps_power}, note
\begin{eqnarray*}
\alpha_* &=~ \ol{f}_* \circ \eps^2_* \circ f_* ~=~ \ol{f}_* \circ \id \circ f_* &=~ \id : \cS^s_\TOP(M) \longra \cS^s_\TOP(M)\\
\alpha^{d^2} &\simeq~ \ol{f} \circ \eps^{2d^2} \circ f ~\simeq~ \ol{f} \circ \id \circ f &\simeq~ \id : M \longra M.
\end{eqnarray*}
Then, by Ranicki's composition formula for simple structure groups \cite{Ranicki_StructuresComposition}, note
\[
d^2 [M,\alpha] ~=~ \sum_{j=0}^{d^2-1} [M,\alpha] ~=~ \sum_{j=0}^{d^2-1} (\alpha_*)^j [M,\alpha] ~=~ [M,\alpha^{d^2}] ~=~ [M,\id] ~=~ 0 ~\in~ \cS^s_\TOP(M).
\]
By Equation~\eqref{eqn:forget} and Corollary~\ref{cor:h}, $\cS^s_\TOP(M) \iso \cS^s_\TOP(X_{d,q})$ is a sum of copies of $\Z/2$ and $\Z$.
So, since $d$ is odd, we must have $[M,\alpha]=0$.
That is, $\alpha$ is $s$-bordant to the identity.
Therefore, by the $s$-cobordism theorem, $\alpha$ is homotopic to a self-homeomorphism $\delta$.
\end{proof}

%------------------------------------------------------------------------------
\section{Classification of homeomorphism types}\label{sec:homeomorphism}

We resume with the calculation of the isotropy subgroups $SI(M)$ from Proposition~\ref{prop:hs}.
Understood in the context of an abelian group $A$ with involution ${}^*$, we consider subgroups
\begin{eqnarray*}
(-1)^n\text{-symmetrics} &:=& \{ a \in A \mid a = (-1)^n a^* \}\\
(-1)^n\text{-evens} &:=& \{ b + (-1)^n b^* \mid b \in A \}.
\end{eqnarray*}
Furthermore, for use later, we abbreviate symmetrics $:= (+1)\text{-symmetrics}$ and evens $:= (+1)\text{-evens}$ and skew-symmetrics $:= (-1)\text{-symmetrics}$ and skew-evens $:= (-1)\text{-evens}$.

\begin{prop}\label{prop:SI}
Let $M$ be a closed connected topological manifold of dimension $n \geqslant 4$.
If $n=4$ then assume $\pi_1 M$ is good in the sense of Freedman--Quinn \cite{FQ_book}.
\begin{enumerate}
\item
With respect to the standard involution on $\Wh_1(\pi_1 M)$ given by $(g \longmapsto g^{-1})$,
\[
(-1)^n\text{-evens} ~\leqslant~  SI(M) ~\leqslant~ (-1)^n\text{-symmetrics}.
\]
Hence $SI(M)/(-1)^n\text{-evens} \leqslant \wh{H}^n(C_2;\Wh_1(\pi_1 M))$, which is a sum of copies of $\Z/2$.

\item
This quotient is expressible in structure groups (add by stacking in the $I$-coordinate):
\[\begin{CD}
\Cok\left( \cS_\TOP^s(M \times I) \to \cS_\TOP^h(M \times I) \right) @>{~\tors~}>{\iso}> \displaystyle\frac{SI(M)}{(-1)^n\text{-evens}}.
\end{CD}\]
\end{enumerate}
\end{prop}

This quantification generalizes a specific argument given by Jahren--Kwasik \cite[\S7]{JK_hcob}.
Our structure sets are `$\rel \bdry$' (homeomorphism on the unspecified boundary \cite[\S0]{Wall_book}).

\begin{proof}[Proof of Proposition~\ref{prop:SI}(1)]
Let $\alpha \in SI(M)$.
There is a strongly inertial $h$-cobordism $(W;M,M')$ such that $\alpha = \tau(W \thra M)$.
By the composition formula \cite[Lemma~7.8]{Milnor_hcob},
\[
0 ~=~ \tau(\id_M) ~=~ \tau(M \hookra W \thra M)
~=~ \tau(W \thra M) + (W \thra M)_* \tau(M \hookra W). 
\]
Next, by Milnor duality \cite[\S10]{Milnor_hcob}, note
\[
\tau(M' \hookra W) ~=~ (-1)^n \tau(M \hookra W)^*.
\]
Finally, since the $h$-cobordism is strongly inertial, by Chapman's topological invariance of Whitehead torsion \cite{Chapman}, by the composition formula again, and by substitution, note
\begin{eqnarray*}
0 &=& \tau(M' \hookra W \thra M) ~=~ \tau(W \thra M) + (W \thra M)_* \tau(M' \hookra W)\\
&=& \alpha + (-1)^n (W \thra M)_* \tau(M \hookra W)^* ~=~ \alpha - (-1)^n \alpha^*.
\end{eqnarray*}
Thus $SI(M) \leqslant (-1)^n\text{-symmetrics}$ in $\Wh_1(\pi_1 M)$.

Let $\beta \in \Wh_1(\pi_1 M)$.
There exists an $h$-cobordism $(W';M,M'')$ with $\beta = \tau(W' \thra M)$.
Consider the \emph{untwisted double} $W := W' \cup_{M''} -W'$.
For avoid confusion, we denote $\bdry W =: -M_0 \sqcup M_1$ with the canonical homeomorphisms $M_i \approx M$ understood.
Note $(W;M_0,M_1)$ is a strongly inertial $h$-cobordism, since $(W' \thra M'' \hookra W')$ is homotopic to the identity:
\begin{eqnarray*}
(M_1 \hookra W \thra M_0) &=& (M_1 \hookra -W' \thra M'' \hookra W' \thra M_0)\\
&\simeq& (M_1 \hookra -W' \stackrel{flip}{\approx} W' \thra M_0)\\
&\simeq& (M_1 \stackrel{\id}{\approx} M_0).
\end{eqnarray*}
Using the above techniques, this doubled $h$-cobordism has Whitehead torsion
\begin{eqnarray*}
\tau(W \thra M_0) &=& \tau(W \thra W' \thra M_0)\\
&=& \tau(W' \thra M_0) + (W' \thra M_0)_* \tau(W \thra W')\\
&=& \beta + (M'' \hookra W' \thra M_0)_* \tau(-W' \thra M'')\\
&=& \beta + (-1)^n \tau(-W' \thra M_1)^*\\
&=& \beta + (-1)^n \beta^*.
\end{eqnarray*}
In the third step, we could excise $\mathring{W'}$ since $W \thra W'$ is the identity on $W'$, whose mapping cone is consists of elementary expansions.
Thus $SI(M) \geqslant (-1)^n\text{-evens}$ in $\Wh_1(\pi_1 M)$.
\end{proof}

\begin{proof}[Proof of Proposition~\ref{prop:SI}(2)]
Let $f: (W;M_0,M_1) \longra M \times (I;0,1)$ be a homotopy equivalence of manifold triads such that the restriction $\bdry f: \bdry W \longra M \times \bdry I$ is a homeomorphism.
Since $f: W \to M \times I$ represents the retraction $W \thra M_0$, the $h$-cobordism $(W;M_0,M_1)$ is strongly inertial.
Then, assuming the identification $\bdry_0 f: M_0 \longra M$, we have
\[
\tau(f) ~=~ \tau(W \thra M_0) ~\in~ SI(M). 
\]

Now suppose $F: (V;W,W') \longra M \times I \times (I;0,1)$ is an $h$-bordism, relative to $M \times \bdry I \times I$, from $f$ to another such homotopy equivalence $f': (W'; M_0',M_1') \longra M \times (I;0,1)$ of triads.
By the composition formula \cite[Lemma~7.8]{Milnor_hcob}, note:
\begin{eqnarray*}
\tau(M_0 \hookra W \hookra  V) &=& \tau(W \hookra V) + (W \hookra V)_* \tau(M_0 \hookra W)\\
\tau(M_0' \hookra W' \hookra V) &=& \tau(W' \hookra V) + (W' \hookra V)_* \tau(M_0' \hookra W').
\end{eqnarray*}
As above, $\tau(f') = \tau(W' \thra M_0')$.
Since $\tau(\id_M)=0$, by \cite[Lemma~7.8]{Milnor_hcob} again, note:
\begin{eqnarray*}
\tau(M_0 \hookra W) &=& -(M_0 \hookra W)_* \tau(f)\\
\tau(M_0' \hookra W') &=& -(M_0' \hookra W')_* \tau(f').
\end{eqnarray*}
By Milnor duality \cite[\S10]{Milnor_hcob}, note
\begin{eqnarray*}
\tau(W' \hookra V) &=& (-1)^{n+1} \tau(W \hookra V)^*.
\end{eqnarray*}
Then, since $M_0 \approx M_0'$ and since $(M_0 \hookra W \hookra V)$ is homotopic to $(M_0' \hookra W \hookra V)$, note:
\begin{eqnarray*}
\tau(W \hookra V) - (M_0 \hookra V)_* \tau(f) &=& (-1)^{n+1}\tau(W \hookra V)^* - (M_0' \hookra V)_* \tau(f')\\
\tau(f) - \tau(f') &=& (M_0 \hookra V)_*^{-1} (1+(-1)^n *) \tau(W \hookra V).
\end{eqnarray*}
Thus we obtain a well-defined homomorphism of abelian groups, where addition in this relative structure set is given by stacking homotopy equivalences in the $I$-coordinate:
\[\begin{CD}
\cS_\TOP^h(M \times I) @>{~\tors~}>> \displaystyle\frac{SI(M)}{(-1)^n\text{-evens}} @. ~;~ [f] \longmapsto [\tau(f)].
\end{CD}\]

Let $\alpha \in SI(M)$.
There is an $h$-cobordism $(W;M,M')$ with torsion $\tau(W \thra M)=\alpha$ such that $(M' \hookra W \thra M)$ is homotopic to a homeomorphism.
First mapping $W \thra M \times \{\frac{1}{2}\}$, and then applying the Homotopy Extension Property with regard to a choice of above homotopy to a homeomorphism $M' \to M$ and a choice of homotopy of $(M \hookra W \thra M)$ to the identity on $M$, we obtain a homotopy equivalence $f: (W;M,M') \longra M \times (I;0,1)$ such that $\bdry f: \bdry W \longra M \times \bdry I$ is the prescribed homeomorphism and $f: W \longra M \times I$ represents $W \thra M$.
Then $[f] \in \cS_\TOP^h(M \times I)$ and $\tau(f) = \alpha$.
Therefore $\tors$ is surjective.

Finally, $\tors[f]=0$ if and only if $f: W \longra M \times I$ is $h$-bordant (as made in Proof~\ref{prop:SI}(1)) to a simple homotopy equivalence.
Thus the kernel of $\tors$ is the image of $\cS_\TOP^s(M \times I)$.
\end{proof}

The homotopy invariance of the subgroup $SI(X) \leqslant \Wh_1(\pi_1 X)$ is now a corollary.

\begin{proof}[Proof of Theorem~\ref{thm:SI}]
The function $\tors$ is a homomorphism with respect to Ranicki's abelian group structure on the structure sets.
This follows from the commutative diagram with exact rows (using Proposition~\ref{prop:SI} and \cite[Theorem~18.5]{Ranicki_TopMan}):
\[\begin{tikzcd}
\cS_\TOP^s(X \times I) \rar \dar{\approx}[swap]{s} & \cS_\TOP^h(X \times I) \rar[two heads]{\tors} \dar{\approx}[swap]{s} & SI(X) / (-1)^n\text{-evens} \dar[tail]\\
\cS_{n+2}^s(X \times I) \rar & \cS_{n+2}^h(X \times I) \rar{\tors} & \wh{H}^n(C_2; \Wh_1(\pi_1 X)).\\
L_{n+2}^s(X \times I) \rar \uar{\bdry} & L_{n+2}^h(X \times I) \uar{\bdry} \rar{\tors} & \wh{H}^{n+2}(C_2; \Wh_1(\pi_1 X)) \uar{}[swap]{\iso}
\end{tikzcd}\]
The bottom two squares consist of \emph{homotopy-invariant} functors from the category of spaces to the category of abelian groups; that is, if continuous functions of spaces are homotopic, then these functors induce equal homomorphisms of abelian groups.

Consider the homotopy class of any continuous function $f: M \longra X$, which induces a homomorphism $f_*: \Wh_1(\pi_1 M) \longra \Wh_1(\pi_1 X)$.
By the functoriality of the upper-right corner of the diagram, the induced map $f_*: \wh{H}^n(C_2; \Wh_1(\pi_1 M)) \longra \wh{H}^n(C_2; \Wh_1(\pi_1 X))$ restricts to a map $f_*: SI(M)/(-1)^n\text{-evens} \longra SI(X)/(-1)^n\text{-evens}$ of subgroups.
Hence the induced map $f_*: \Wh_1(\pi_1 M) \longra \Wh_1(\pi_1 X)$ restricts to a map $f_*: SI(M) \longra SI(X)$.
If $f$ is a homotopy equivalence, then all of these induced maps are isomorphisms.
\end{proof}

The following proposition is not original; it is merely a record.
Recall $X_{d,q} = S^1 \times L_{d,q}^{2k-1}$.

\begin{prop}\label{prop:Wh}
Let $d > 1$ be a square-free odd integer.
Select an integer $q$ coprime to $d$.
\begin{enumerate}
\item 
There is a canonical identification
\[
\Wh_1(\pi_1 X_{d,q}) ~=~ \Wh_1(C_d) ~\oplus~ \Wh_0(C_d).
\]

\item
The standard involution $(g \longmapsto g^{-1})$ on $\Wh_1(\pi_1 X_{d,q})$ restricts to the standard involution on $\Wh_1(C_d)$ and to negative the standard involution on $\Wh_0(C_d)$.

\item
Furthermore, with respect to these restricted involutions:
\[
\frac{\Wh_1(C_d)}{\text{symmetrics}} ~=~ 0 \qquad\text{and}\qquad
\frac{\Wh_0(C_d)}{\text{skew-evens}} ~=~ H_0(C_2;\Wh_0(C_d)).
\]
\end{enumerate}
\end{prop}

\begin{proof}
Part~(1) is the fundamental theorem of algebraic $K$-theory \cite[XII:7.3, 7.4b]{Bass_book} combined with the vanishing of $NK_1(Z[C_d])$ for $d$ square-free \cite{Harmon}.

Part~(2) is the analysis of the restriction of the overall involution done in \cite[p21]{Ranicki_AlgL3}.

For Part~(3), by Remark~\ref{rem:Milnor_mistake}, the group-ring involution $(g \longmapsto g^{-1})$ on $\Z[C_d]$ induces the identity on $\Wh_1(C_d)$.
Therefore $\Wh_1(C_d)/\text{symmetrics} = 0$.
The assertion about $\Wh_0(C_d)$ is simply the definition of $H_0(C_2;\Wh_0(C_d))$.
\end{proof}

\begin{cor}\label{cor:SI}
Let $d > 1$ be square-free odd, select an integer $q$ coprime to $d$, let $k>1$.
Let $M$ be any closed topological manifold in the homotopy type of $X_{d,q}$.
We can identify
\[
\frac{\Wh_1(\pi_1 X_{d,q})}{SI(M)} ~=~ H_0(C_2;\Wh_0(C_d)).
\]
\end{cor}

\begin{proof}
By Theorem~\ref{thm:SI}, we have $SI(M) = SI(X_{d,q})$ as subgroups of $\Wh_1(\pi_1 X_{d,q})$.

The surgery exact sequence for $X_{d,q} \times I \;\rel \bdry$ admits forgetful maps of decorations.
Consider the commutative diagram with exact rows, which we write schematically:
\begin{equation}\label{eqn:4lemma}\begin{tikzcd}
H_{2k+2} \rar \dar[equals] & L_{2k+2}^s \rar \dar & \cS^s \rar \dar & H_{2k+1} \rar \dar[equals] & L_{2k+1}^s \dar\\
H_{2k+2} \rar & L_{2k+2}^h \rar & \cS^h \rar & H_{2k+1} \rar & L_{2k+1}^h.
\end{tikzcd}\end{equation}
By Ranicki's version of Shaneson's thesis \cite{Ranicki_AlgL2}, Bak's vanishing result \cite{Bak_oddL}, and Bak--Kolster's vanishing result \cite[Corollary~4.7]{BK}, note the computations:
\begin{eqnarray*}
L_{2k+2}^s(C_\infty \times C_d) &=~ L_{2k+2}^s(C_d) \oplus L_{2k+1}^h(C_d) &=~ L_{2k+2}^s(C_d)\\
L_{2k+2}^h(C_\infty \times C_d) &=~ L_{2k+2}^h(C_d) \oplus L_{2k+1}^p(C_d) &=~ L_{2k+2}^h(C_d)\\[\medskipamount]
L_{2k+1}^s(C_\infty \times C_d) &=~ L_{2k+1}^s(C_d) \oplus L_{2k}^h(C_d) &=~ L_{2k}^h(C_d)\\
L_{2k+1}^h(C_\infty \times C_d) &=~ L_{2k+1}^h(C_d) \oplus L_{2k}^p(C_d)  &=~ L_{2k}^p(C_d).
\end{eqnarray*}
Substituting, we may now consider the following commutative diagram of groups:
\begin{equation}\label{eqn:9lemma}\begin{tikzcd}
{} & 0 \dar & 0 \dar & 0 \dar & {}\\
0 \rar & L_{2k+2}^s(C_d) / H_{2k+2} \rar \dar & \cS^s \rar \dar & H_{2k+1} / L_{2k}(1) \rar \dar[equals] & 0\\
0 \rar & L_{2k+2}^h(C_d) / H_{2k+2} \rar \dar & \cS^h \rar \dar & H_{2k+1} / L_{2k}(1) \rar \dar & 0\\
0 \rar & \wh{H}^{2k+2}(C_2;\Wh_1(C_d)) \rar \dar & \cS^h / \cS^s \rar \dar & 0 \rar \dar & 0\\
{} & 0 & 0 & 0 & {}
\end{tikzcd}\end{equation}

Clearly, the right column of \eqref{eqn:9lemma} is exact.
Next, Bass--Milnor--Serre showed that $\Wh_1(C_d)$ is a free abelian group and that the group-ring involution $(g \longmapsto g^{-1})$ on $\Z[C_d]$ induces the identity on $\Wh_1(C_d)$; refer to Remark~\ref{rem:Milnor_mistake}.
Then the subgroup of skew-symmetrics in $\Wh_1(C_d)$ is zero.
So $\wh{H}^{2k+3}(C_2; \Wh_1(C_d)) = 0$.
Recall the vanishing result above: $L_{2k+1}^s(C_d)=0$.
Therefore, by the Rothenberg sequence, the left column of \eqref{eqn:9lemma} is exact.
Then, finally, a diagram chase in \eqref{eqn:4lemma} shows that the middle column of \eqref{eqn:9lemma} is exact.

The generalized homology of a space cross a circle admits a canonical decomposition:
\[
H_{2k+1} ~=~ H_{2k+1}(X_{d,q};\bL\!\gens{1}) ~=~ H_{2k+1}(L_{d,q}^{2k-1};\bL\!\gens{1}) ~\oplus~ H_{2k}(L_{d,q}^{2k-1};\bL\!\gens{1}).
\]
By naturality, the assembly map $H_{2k+1} \longra L_{2k+1}^{s,h}$ for $X_{d,q}$ is the direct sum of the assembly maps $H_{2k+1} \longra L_{2k+1}^{s,h} = 0$ and $H_{2k} = L_{2k}(1) \longra L_{2k}^{h,p}$ for $L_{d,q}^{2k-1}$ by \eqref{eqn:evenhomology}.
So the kernel of the assembly map $H_{2k+1} \longra L_{2k+1}^{s,h}$ for $X_{d,q}$ is the summand $H_{2k+1}(L_{d,q}^{2k-1};\bL\!\gens{1}) \iso H_{2k+1}/L_{2k}(1)$.
Therefore, by exactness of rows in \eqref{eqn:4lemma}, the top and middle rows of \eqref{eqn:9lemma} are exact.

So, by the Nine Lemma, the bottom row of \eqref{eqn:9lemma} is exact.
Then, by Proposition~\ref{prop:SI},
\[
\frac{SI(X_{d,q})}{\text{evens}} = \wh{H}^{2k+2}(C_2;\Wh_1(C_d)) = \frac{\text{symmetrics in } \Wh_1(C_d)}{\text{evens in } \Wh_1(C_d)}.
\]
Therefore, we obtain the formula
\begin{equation}\label{eqn:SI}
SI(X_{d,q}) ~=~ \text{symmetrics in } \Wh_1(C_d) ~\oplus~ \text{skew-evens in } \Wh_0(C_d).
\end{equation}
The calculation of $\Wh_1(X_{d,q}) / SI(X_{d,q})$ now follows from Proposition~\ref{prop:Wh}.
\end{proof}

\begin{rem}\label{rem:summary}
Proposition~\ref{prop:hs}, Corollary~\ref{cor:h}, and Corollary~\ref{cor:SI} produce a based bijection
\[\begin{CD}
\Z^{(d-1)/2} ~\times~ H_0(C_2;\Wh_0(C_d)) @>>{\approx}> \cS^{h/s}_\TOP(X_{d,q}).
\end{CD}\]
\end{rem}

%------------------------------------------------------------------------------
\section{Computation of the action of the group of self-equivalences}\label{sec:actions}

For any topological space $Z$, write $\Map(Z)$ for the topological monoid of continuous self-maps $Z \longra Z$.
Recall that $\hMod(Z) \subset \pi_0 \Map(Z)$ is the group of homotopy classes of self-homotopy equivalences.
A pair $(X_1,X_2)$ of based topological spaces satisfies the \emph{Induced Equivalence Property} if $[f] \in \hMod(X_1 \times X_2)$ implies $[p_j \circ f \circ i_j] \in \hMod(X_j)$ for both $j=1,2$, with based inclusion $i_j: X_j \longra X_1 \times X_2$ and projection $p_j: X_1 \times X_2 \longra X_j$.
We slightly simplify the following result of P\,I~Booth and P\,R~Heath \cite[Corollary~2.8]{BH}.
Write $[-,-]_0$ for the set of the based homotopy classes of maps preserving basepoint.

\begin{thm}[Booth--Heath]\label{thm:hMod}
Let $X$ be a connected CW complex equipped with a co-$H$-space structure.
Let $Y$ be a based connected CW complex with $[Y,X]_0 = 0 = [X \wedge Y, X]_0$.
If $(X,Y)$ satisfies the Induced Equivalence Property, there is a split exact sequence of groups:
\[\begin{tikzcd}
1 \rar & {[X,\Map(Y)]_0} \rar & \hMod(X \times Y) \rar & \hMod(X) \times \hMod(Y) \rar & 1.
\end{tikzcd}\]
\end{thm}

\begin{cor}\label{cor:hMod_circle}
Let $Y$ be a nonempty connected CW complex.
Suppose that $\pi_1(Y)$ is finite.
Then there is a natural decomposition of groups:
\[
\hMod(S^1 \times Y) ~=~ \pi_1 \Map(Y) \rtimes (\hMod\,S^1 \times \hMod\,Y).
\]
Hence, each element of $\hMod(S^1 \times Y)$ is \emph{splittable:} it restricts to a self-equivalence of $1 \times Y$.
\end{cor}

This is false without the hypothesis, since $\hMod(S^1 \times S^1) = GL_2(\Z) \not\iso \Z \rtimes (\{\pm 1\} \times \Z)$.

\begin{proof}[Proof of Corollary~\ref{cor:hMod_circle}]
The circle $X=S^1$ is a co-$H$-space, and it is a model of $K(\Z,1)$.
Note $[Y,X]_0 = H^1(Y;\Z) = 0$ and $[X \wedge Y, X]_0 = H^1(S^1 \wedge Y;\Z) \iso \wt{H}_0(Y;\Z) = 0$.
By Theorem~\ref{thm:hMod}, it remains to show that $(S^1,Y)$ satisfies the Induced Equivalence Property.
Let $f: S^1 \times Y \longra S^1 \times Y$ be a based homotopy equivalence.

On the one hand, to prove that $p_1 \circ f \circ i_1: S^1 \longra S^1$ is a homotopy equivalence, we must show that induced map on the Hopfian group $\pi_1(S^1)=C_\infty$ is surjective.
Since $f_\#$ is surjective, there exists $(a,b) \in \pi_1(S^1) \times \pi_1(Y)$ such that $f_\#(a,b)=(t,1)$, where $t$ generates $\pi_1(S^1)$.
Then, since $\Hom(\pi_1 Y, \pi_1 S^1) = 1$, note $(p_1)_\#(f_\#(1,b))=1$.
So $(p_1)_\#(f_\#(a,1))=t$.

On the other hand, $f$ induces an isomorphism on $\pi_n(S^1 \times Y)=\pi_n(Y)$ for all $n > 1$.
Since $Y$ is a CW complex, by the Whitehead theorem, it remains to show that $p_2 \circ f \circ i_2$ is injective on the co-Hopfian group $\pi_1(Y)$.
For all $b \in \pi_1(Y)$, recall $(p_1)_\#(f_\#(1,b))=1$.
Then $(p_2 \circ f \circ i_2)_\#(b)=1$ if and only if $f_\#(1,b)=1$, if and only if $b=1$ since $f_\#$ is injective.
\end{proof}

\begin{rem}
The corollary below is parallel to $p=2$; Jahren--Kwasik \cite[3.5]{JK} showed
\[
\hMod(S^1 \times \mathbb{RP}^{2k-1}) ~=~
\begin{cases}
C_2 \times (C_2)^2 & \text{if } k \equiv 0 \!\!\!\pmod{2}\\
C_2 \times C_4 & \text{if } k \equiv 1 \!\!\!\pmod{2}
\end{cases}
\;\times\; (C_2 \times C_2).
\]
Unlike below, the first factor (the $C_2$ on the left) is not represented by a diffeomorphism.
The very last $C_2$ factor is represented by the diffeomorphism of $\RP^n$ that reflects in $\RP^{n-1}$.
\end{rem}

\begin{cor}\label{cor:hMod_lens}
Let $d >1 $ be odd, $q$ coprime to $d$, and $k>1$.
We have a metabelian group
\[
\hMod(S^1 \times L_{d,q}^{2k-1}) ~=~ A \rtimes (C_2 \times B),
\]
where $A$ is abelian of order $2d^2$, and where $B$ is the exponent $e := \gcd(2k,\varphi(d))$ subgroup of $\Aut(C_d)$.\footnote{Classically, it is known that $\Aut(C_d)$ has order $\varphi(d)$. If $d$ is an odd-prime power, then $\Aut(C_d)$ is cyclic.  Conversely, $\Aut(C_d)$ contains a product of $C_2$'s, one for each odd-prime factor of $d$, such as: $\Aut(C_{15})=C_2 \times C_4$.}
Furthermore, the subgroup $A \rtimes C_2$ is generated by the three diffeomorphisms
\begin{align*}
\rho ~:~ (z,[u]) \longmapsto~ & (z,[z u_1: u_2: \ldots : u_k])\\
\eps ~:~ (z,[u]) \longmapsto~ & (z,[z^{q/d} u_1: z^{1/d} u_2: \ldots : z^{1/d} u_k])\\
{}^- \times \id_{L^n} ~:~ (z,[u]) \longmapsto~ & (\ol{z},[u]).
\end{align*}
\end{cor}

\begin{proof}
Since the fundamental group $\pi_1(L^n)=C_d$ is finite, by Corollary~\ref{cor:hMod_circle}, we have
\[
\hMod(S^1 \times L^n) ~=~ \pi_1\Map(L^n) \rtimes (\hMod\,S^1 \times \hMod\,L^n).
\]
The subgroup $\hMod(S^1)$ is generated by the homotopy class of the diffeomorphism ${}^- \times \id_{L^n}$.
Since $d$ is odd, by \cite[(29.5)]{Cohen}, any homotopy equivalence $h: L^n \longra L^n$ is classified uniquely by the induced automorphism $h_\#: s \longmapsto s^a$ on $\pi_1(L^n)$ where $a^k \equiv \deg(h) \pmod{d}$ and $\deg(h) = \pm 1$; any $a$ with $a^k \equiv \pm 1 \pmod{d}$ is induced by an equivalence $h_a: L^n \longra L^n$.
That is, since $a^k \equiv \pm 1 \pmod{d}$ if and only if $a^{2k} \equiv 1 \pmod{d}$, the homomorphism $\#: \hMod(L^n) \longra \Out(\pi_1 L^n) = \Out(C_d)$ is injective with image the subgroup $B$ of exponent $e$.

Consider the fibration sequence $\Map_0(L^n) \longra \Map(L^n) \longra L^n$, where $\Map_0  \subseteq \Map$ is the topological submonoid of basepoint-preserving self-maps.
Since $\pi_2(L^n)=0$, and since any unbased homotopy between two based self-maps of a connected CW complex is relatively homotopic to a based homotopy, there is an exact sequence of abelian groups:
\begin{equation}\label{eqn:based}\begin{tikzcd}
1 \rar & \pi_1\Map_0(L^n) \rar & \pi_1\Map(L^n) \rar & \pi_1(L^n) \rar & 1.
\end{tikzcd}\end{equation}
On the one hand, Hsiang--Jahren \cite[Proposition~3.1]{HJ} showed that the forgetful map $\pi_1\Diff_0(L^n) \longra \pi_1\Map_0(L^n)$ is surjective with image of order $2d$ generated by the based homotopy class $[\rho]_0$ of the diffeomorphism $\rho$.
On the other hand, since $\eps_\#(t)=ts$, the unbased homotopy class $[\eps]$ of the diffeomorphism $\eps$ maps to the generator $s$ of $\pi_1(L^n)$.
Therefore $\pi_1\Map(L^n)$ is an abelian group of order $2d^2$ generated by $[\rho]_0$ and $[\eps]$.
\end{proof}

To find $\cM^{h/s}_\TOP(X_{d,q})$, we now compute the action of the group $\hMod(X_{d,q})$ on $\cS^{h/s}_\TOP(X_{d,q})$.

\begin{proof}[Proof of Theorem~\ref{thm:homeo}]
First, we show the order $d^2$ subgroup of $\hMod(X_{d,q})$ acts trivially.
By Proof~\ref{cor:hMod_lens}, this subgroup is generated by the classes $[\rho^2]$ and $[\eps^2]$ of diffeomorphisms.
Let $[M,f] \in \cS^{h/s}_\TOP(X_{d,q})$; write $\ol{f}: X_{d,q} \longra M$ for a homotopy inverse of $f: M \longra X_{d,q}$.
For any element $[\phi] \in \hMod(X_{d,q})$, consider the \emph{pullback} $f^*[\phi] := [\ol{f} \circ \phi \circ f] \in \hMod(M)$.
Recall, by Proposition~\ref{prop:transitive}, that each pullback $f^*[\eps^2]$ is represented by a homeomorphism.
Thus $[\eps^2]$ acts trivially on the hybrid structure set $\cS^{h/s}_\TOP(X_{d,q})$.

The overall argument for $[\rho^2]$ is similar but slightly simpler to that of $[\eps^2]$ in Section~\ref{sec:eps}.
By the composition formula for Whitehead torsion \cite[Lemma~7.8]{Milnor_hcob}, and since $\rho_\# = \id$,
\[
\tau(f^*\rho)
~=~ \tau(\ol{f}) + \ol{f}_* (\tau(\rho) + \rho_* \tau(f))
~=~ -f^{-1}_*\tau(f) + f^{-1}_* (0 + \tau(f))
~=~ 0 ~\in~ \Wh_1(\pi_1 M).
\]
So $[M,f^*\rho] \in \cS^s_\TOP(M)$.
Similar to Proposition~\ref{prop:Shaneson}, there is a direct sum decomposition
\[
\cS^s_\TOP(X_{d,q}) ~\iso~ \cS^s_\TOP(I \times L^n) ~\oplus~ \cS^h_\TOP(L^n).
\]
Since $\rho$ restricts to $\id$ on $1 \times L^n \subset S^1 \times L^n$, there is an induced commutative diagram
\[\begin{tikzcd}
0 \rar & \cS^s_\TOP(I \times L^n) \rar{glue} \dar{(\rho|)_*} & \cS^s_\TOP(X_{d,q}) \rar{split} \dar{\rho_*} & \cS^h_\TOP(L^n) \rar \dar[dashed]{[\rho_*]} & 0\\
0 \rar & \cS^s_\TOP(I \times L^n) \rar{glue} & \cS^s_\TOP(X_{d,q}) \rar{split} & \cS^h_\TOP(L^n) \rar & 0.
\end{tikzcd}\]
The decomposition is compatible with those of $L^s_*(\pi_1 X_{d,q})$ and $H_*(X_{d,q}; \bL\!\gens{1})$, inducing
\[\begin{tikzcd}
0 \rar & \wt{L}^h_{2k}(C_d) \rar \dar{[(\rho_\#)_*] = \id} & \cS^h_\TOP(L^n) \rar \dar{[\rho_*]} & H_{2k-1}(L^n;\bL\!\gens{1}) \rar \dar & 0\\
0 \rar & \wt{L}^h_{2k}(C_d) \rar & \cS^h_\TOP(L^n) \rar & H_{2k-1}(L^n;\bL\!\gens{1}) \rar & 0.
\end{tikzcd}\]
Recall from Proof~\ref{lem:projective} that $H_{2k-1}(L^n;\bL\!\gens{1})$ is an abelian group annihilated by a power of $d$.
A similar argument to that proof shows that $\cS^h_\TOP(L^n)$ has no `$d$-torsion'.\footnote{This lack of `$d$-torsion' is true for the $h$-structure group, despite that $\wt{L}^h_{2k}(C_d)$ may now have some $2$-torsion.}
So $[\rho_*] = \id$ on $\cS^h_\TOP(L^n)$.
But $\cS^s_\TOP(I \times L^n)=0$ by Lemma~\ref{lem:crossI}.
Therefore $\rho_* = \id$ on $\cS^s_\TOP(X_{d,q})$.
Then
\begin{eqnarray*}
(f^*\rho^2)_* &=~ \ol{f}_* \circ (\rho^2)_* \circ f_* ~=~ \ol{f}_* \circ \id \circ f_* &=~ \id : \cS^s_\TOP(M) \longra \cS^s_\TOP(M)\\
(f^*\rho^2)^d &\simeq~ \ol{f} \circ \rho^{2d} \circ f ~\simeq~ \ol{f} \circ \id \circ f &\simeq~ \id : M \longra M.
\end{eqnarray*}
Then, by Ranicki's composition formula for simple structure groups \cite{Ranicki_StructuresComposition}, note
\[
d [M,f^*\rho^2] ~=~ \sum_{j=0}^{d-1} [M,f^*\rho^2] ~=~ \sum_{j=0}^{d-1} (f^*\rho^2)_*^j [M,f^*\rho^2] ~=~ [M,(f^*\rho^2)^d] ~=~ 0 \in \cS^s_\TOP(M).
\]
By Equation~\eqref{eqn:forget} and Corollary~\ref{cor:h}, $\cS^s_\TOP(M) \iso \cS^s_\TOP(X_{d,q})$ is a sum of copies of $\Z/2$ and $\Z$.
So $[M,f^*\rho^2]=0$ since $d$ is odd.
That is, $f^*\rho^2$ is $s$-bordant to $\id$.
By the $s$-cobordism theorem, $f^*\rho^2$ is homotopic to a homeomorphism.
Thus $[\rho^2]$ acts trivially on $\cS^{h/s}_\TOP(X_{d,q})$.
Therefore, from Corollary~\ref{cor:hMod_lens}, the order $d^2$ subgroup of $\hMod(X_{d,q})$ acts trivially.

Now, this induces a left action of the quotient group $C_2 \times C_2 \times B$ on the set $\cS^{h/s}_\TOP(X_{d,q})$.
Thus, by Remark~\ref{rem:summary}, we are done, since this group has order $4e = 8 \gcd(k,\varphi(d)/2)$.
\end{proof}

\begin{rem}\label{rem:cardinality}
Let $p\neq 2$ be prime.
This quotient group \emph{does not act with uniform isotropy,} unlike the order $p^2$ subgroup.
To conclude, we discuss the three generators of $C_2 \times C_2 \times C_e$.
\begin{enumerate}
\item
The above methods show that post-composition with $\rho^p$ is the identity on the $h$-cobordism structure group.
There may be a `cross-effect' on the $s$-cobordism structure group, that is, a nonzero component of $\rho^p_*$ from the free part of $\cS^s_\TOP(X_{p,q})$ to the 2-torsion part.
The author is unaware of the effect within $H_0(C_2;\Cl_p)$-orbits.

\item
Since complex conjugation ${}^-$ reverses orientation on the symmetric Poincar\'e complex $\sigma^*(S^1) \in L^1(C_\infty)$, post-composition with the diffeomorphism ${}^- \times \id_{L_{p,q}}$ is negation\footnote{\cite[Lemma~3.7]{JK} falsely implies that ${}^- \times \id_{\RP^n}$ induces the identity on $\cS_\TOP(S^1 \times \RP^n)$, rather than negation.  The proof's error is: Ranicki's $\bL^\cdot$-orientation of a manifold is preserved by tangential homotopy equivalences. Call a manifold \emph{$w_1$-oriented} if an orientation is chosen on the $\Ker(w_1)$-cover \cite[p216]{Wall_PD}. The correction is: the $\bL^\cdot$-orientation of a $w_1$-oriented manifold is preserved by $w_1$-oriented tangential homotopy equivalences \cite[16.16, App.~A]{Ranicki_TopMan}.  For example, the diffeomorphism ${}^- \times \id_{\RP^n}$ is tangential with $\mu=+1$ but reverses $w_1$-orientation.} on the $h$-cobordism structure group $\cS^h_\TOP(X_{p,q}) \xleftarrow{\iso} \cS^p_\TOP(L_{p,q}) = \Z^{(p-1)/2}$.
Then ${}^- \times \id_{L_{p,q}}$ must act freely away from the $H_0(C_2;\Cl_p)$-orbit of the basepoint $[X_{p,q},\id]$ of $\cS^{h/s}_\TOP(X_{p,q})$.
But ${}^- \times \id_{L_{p,q}}$ must fix $[X_{p,q},\id]$, since any two homeomorphisms $M \longra X_{p,q}$ are $s$-bordant.\footnote{Suppose there exists $[\alpha] \neq 0 \in H_0(C_2;\Cl_p)$, for example if $p=29$ by Remark~\ref{rem:class_comp}. It is unlikely that ${}^- \times \id_{L_{p,q}}$ fixes $[X_{p,q},\id] \cdot [\alpha]$, since the $h$-cobordism $W_\alpha$ on $X_{p,q}$ with torsion $\alpha \in \Wh_1(C_\infty \times C_p)$ has projection $\alpha \neq 0 \in \Wh_0(C_p) = \Cl_p$.  So the $h$-cobordism is unlikely splittable along $1 \times L_{p,q}$; compare with \cite[6.1, 6.3]{FH_split}.}
So ${}^- \times \id_{L_{p,q}}$ acts non-uniformly on $\cS^{h/s}_\TOP(X_{p,q})$.

\item
Let $a$ be a primitive $e$-th root of unity in the field $\F_p$.
Recall, from Proof~\ref{cor:hMod_lens}, that the homotopy equivalence $h_a: L_{p,q} \longra L_{p,q}$ uniquely induces $s \longmapsto s^a$ on fundamental group.
Note $\id_{S^1} \times h_a: X_{p,q} \longra X_{p,q}$ has zero Whitehead torsion, by the product formula, but the author suspects that $\id_{S^1} \times h_a$ is often non-representable by a homeomorphism of $X_{p,q}$.\footnote{Using a splitting argument along $1 \times L_{p,q}$: if $\id_{S^1} \times h_a$ is homotopic to a homeomorphism, then $h_a$ is $h$-bordant to a homeomorphism, if and only if the Whitehead torsion $\tau(h_a)$ is divisible by two in $\Wh_1(C_p) \iso \Z^{(p-3)/2}$.  Note $h_a$ is homotopic to a homeomorphism if and only if $\tau(h_a)=0$ \cite[\S31]{Cohen}, if and only if $e=2$ \cite[(30.1)]{Cohen}.}
On the other hand, the automorphism of $\cS_\TOP^h(X_{p,q})$ induced by $\id_{S^1} \times h_a$ is identified with the automorphism of $\cS_\TOP^p(L_{p,q}) \iso \Z^{(p-1)/2}$ induced by $h_a$, given by a permutation matrix $\Pi_a$ of order $e/2$ determined by $a$.
Both these issues complicate the systematic use of Ranicki's composition formula:
\[
[(\id_{S^1} \times h_a) \circ (f: M \longra X_{p,q})] ~=~ [\id_{S^1} \times h_a] + \Pi_a [f] ~\in~ \cS_\TOP^h(X_{p,q}) \iso \Z^{(p-1)/2}.
\]
\end{enumerate}
\end{rem}

%------------------------------------------------------------------------------
\subsection*{Acknowledgments}

I wish to thank manifold topologists Jim Davis, Jonathan Hillman, and Shmuel Weinberger for helpful input on Proposition~\ref{prop:crosscircle}.
Also I am grateful to computational number theorists John Miller and Ren\'e Schoof for discussions about Remark~\ref{rem:class_comp}.
Remark~\ref{rem:classical} is a response to Ian Hambleton's questions during my talk at U~Bonn (2014).
I thank Bj\o{}rn Jahren for catching typos in Sections~\ref{sec:hcobordism} and \ref{sec:homeomorphism}.
Finally, the referee found several other typos and also corrected applications of the references (e.g., leading to Remark~\ref{rem:Milnor_mistake}).

%------------------------------------------------------------------------------
\bibliographystyle{alpha}
\bibliography{FreeCl_onS1Sn}

\end{document}